\documentclass[10pt,oneside,reqno]{amsart}

\usepackage{amsmath,amsthm} 
\usepackage{enumerate}
\usepackage{xfrac}          
\numberwithin{equation}{section}
\theoremstyle{definition}
\newtheorem{Definition}{Definition}[section]
\newtheorem{Example}[Definition]{Example}
\newtheorem{Remark}[Definition]{Remark}

\theoremstyle{plain}
\newtheorem{Theorem}[Definition]{Theorem}

\newtheorem{Proposition}[Definition]{Proposition}
\newtheorem{Corollary}[Definition]{Corollary}
\newtheorem{Lemma}[Definition]{Lemma}

\usepackage{amssymb}        
\usepackage{mathrsfs}       
\usepackage{eucal}          
\usepackage{stmaryrd}       

\usepackage{geometry}

\usepackage{hyperref}

\usepackage[usenames,dvipsnames]{xcolor}
\usepackage{tikz}
\usetikzlibrary{arrows, matrix}
\usepackage{tikz-cd}
\usepackage{subfig}         
\usepackage{arydshln}       

\newcommand{\al}{\alpha}

\newcommand{\Ga}{\Gamma}

\newcommand{\ep}{\varepsilon}

\newcommand{\la}{\lambda}

\newcommand{\si}{\sigma}

\newcommand{\N}{\mathbb{N}}
\newcommand{\Z}{\mathbb{Z}}

\newcommand{\R}{\mathbb{R}}
\newcommand{\C}{\mathbb{C}}
\newcommand{\K}{\Bbbk}

\newcommand{\Bbo}{\boldsymbol{1}}
\newcommand{\Bzr}{\boldsymbol{0}}

\newcommand{\Fsl}{\mathfrak{sl}}

\newcommand{\Fm}{\mathfrak{m}}

\newcommand{\CA}{\mathcal{A}}

\newcommand{\CI}{\mathcal{I}}

\newcommand{\CO}{\mathcal{O}}

\newcommand{\scA}{\mathscr{A}}

\DeclareMathOperator{\Aut}{Aut}

\DeclareMathOperator{\Id}{Id}
\renewcommand{\Im}{\operatorname{Im}}

\DeclareMathOperator{\MaxSpec}{MaxSpec}

\DeclareMathOperator{\Specm}{Specm}

\DeclareMathOperator{\Supp}{Supp}

\renewcommand{\hat}{\widehat}
\renewcommand{\tilde}{\widetilde}

\newcommand{\spl}{\text{split}}
\newcommand{\ronto}{\twoheadrightarrow}

\title{Grothendieck rings of towers of twisted generalized Weyl algebras}
\author{Jonas Hartwig \and Daniele Rosso}
\date{\today}

\address{Department of Mathematics, Iowa State University, Ames, IA-50011, USA}
\email{jth@iastate.edu}
\urladdr{http://jth.pw}

\address{Department of Mathematics and Actuarial Science, Indiana University Northwest, Gary, IN-46408, USA}
\email{drosso@iu.edu}
\urladdr{http://pages.iu.edu/~drosso/}

\begin{document}
\maketitle
\begin{abstract}Twisted generalized Weyl algebras (TGWAs) $A(R,\sigma,t)$ are defined over a base ring $R$ by parameters $\sigma$ and $t$, where $\sigma$ is an $n$-tuple of automorphisms, and $t$ is an $n$-tuple of elements in the center of $R$. We show that, for fixed $R$ and $\sigma$, there is a natural algebra map $A(R,\sigma,tt')\to A(R,\sigma,t)\otimes_R A(R,\sigma,t')$. This gives a tensor product operation on modules, inducing a ring structure on the direct sum (over all $t$) of the Grothendieck groups of the categories of weight modules for $A(R,\sigma,t)$. We give presentations of these Grothendieck rings for $n=1,2$, when $R=\C[z]$. As a consequence, for $n=1$, any indecomposable module for a TGWA can be written as a tensor product of indecomposable modules over the usual Weyl algebra. In particular, any finite-dimensional simple module over $\mathfrak{sl}_2$ is a tensor product of two Weyl algebra modules.
\end{abstract}


\section{Introduction}

Generalized Weyl algebras (introduced by Bavula \cite{Bavula1992} and Rosenberg \cite{Rosenberg1995}) and, more generally, twisted generalized Weyl algebras (TGWAs), introduced by Mazorchuck and Turowska (\cite{MT1999}), are a broad family of algebras defined by generators and relations starting from a base ring $R$ and certain parameters $\sigma$ and $t$, where $\sigma\in\Aut(R)^n$ and $t\in Z(R)^n$. Many algebras of interest for ring theory and representation theory can be seen as special cases of this construction, for example (quantized) Weyl algebras in rank $n$ as well as certain quotients of universal enveloping algebras and many others. Modules over TGWAs can be studied in great generality, of particular interest are weight modules, where the base ring $R$ plays the role that the Cartan subalgebra has in the study of representations of a Lie (or Kac-Moody) algebra. These have been studied in \cite{DGO},\cite{MT1999},\cite{MPT2003},\cite{H2006},\cite{H2018}.

However, unlike the case of enveloping algebras, TGWAs in general do not have a natural structure of bialgebras or Hopf algebras and there is no obvious monoidal structure on their module categories (weight or otherwise). In this paper we propose the vision that, instead of considering each TGWA individually, we should group all of them together, for fixed $R$ and $\sigma$ and varying $t$. In fact if $A(R,\si,t)$ is a TGWA, despite the fact that it is not a bialgebra itself, our first main result (Theorem \ref{thm:coprod}) shows that there is a naturally defined algebra map $\Delta:A(R,\sigma,tt')\to A(R,\sigma,t)\otimes_R A(R,\sigma,t')$. Therefore, by taking the restriction functor along $\Delta$, we can define an interesting tensor product operation on the direct sum, over all values of the parameter $t$, of the module categories for the TGWAs.

This construction is reminiscent of the towers of algebras, as formalized by Bergeron and Li in \cite{BL2009}, which is why we use this terminology, although it is a little bit of an abuse because there are some significant differences. For example, our map $\Delta$ goes in the opposite direction, and our algebras are not finite dimensional, though these are not significant concerns. More importantly, $ A(R,\sigma,t)\otimes_R A(R,\sigma,t')$ is not in general a finite rank module over the image of $A(R,\sigma,tt')$, hence taking the induction functor along $\Delta$ does not give well defined operations on the Grothendieck groups of the module categories. Consequently, the only structure that we can define on the sum of all the Grotendieck groups $\bigoplus_t K_0(A(R,\sigma,t)\text{-mod})$ is that of an associative algebra via tensor product and restriction.

Our other main results are giving explicit descriptions, both in terms of a basis with multiplication formulas and in terms of generators and relations, of the algebras we defined, for the cases of TGWAs in rank $1$ and $2$. In rank $1$ we are able to describe also the algebra resulting from the split Grothendieck groups and the map from the split algebra to the non-split one.

It follows from our results that any indecomposable module of a TGWA of rank $1$ can be written as a tensor product of modules for the usual Weyl algebra, which applies in particular to finite dimensional irreducible modules and Verma modules for $\mathfrak{sl}_2$.

The paper is organized as follows:
\begin{itemize}
\item In Section \ref{sec:tgwa} we recall some basic definitions about twisted generalized Weyl algebras that we will need in the paper.
\item In Section \ref{section:towers} we show the existence of the $\Delta$ map, use it to define the associative algebras structure on the sum of all the Grothendieck groups (both non-split and split), and prove some general structural results about these algebras.
\item In Section \ref{sec:line} we explicitly describe the algebras from Section \ref{section:towers} (both non-split and split) in the case of rank $1$, with the base ring being polynomials in one variable and $\sigma$ being a shift. 
\item In Section \ref{sec:cylinder} we give an explicit description of the non-split algebra in the case of rank $2$ over polynomials in one variable. This relies heavily on previous work by the authors in \cite{HarRos16}, \cite{H2018}, \cite{HarRos20}.
\end{itemize}

\subsection{Future Directions}
Our construction is very general, so there are many interesting special cases of these Grothendieck rings that can be approached in the near future. We list a few such possibilities.
\begin{itemize}
\item We have not described here the split algebra for the rank $2$ case. In order for this to be done, it requires giving a characterization of indecomposable weight modules for rank $2$ TGWAs which is a potentially intriguing avenue of research.
\item In the rank $1$ case, we can take $\sigma$ to be an automorphism of finite order instead. This would change the geometry of the $\sigma$-orbit from a line to a circle, and can be accomplished either by taking a base ring with positive characteristic or by scaling by a root of unity. Both simple and indecomposable weight modules for such TGWAs are relatively well understood (see \cite{DGO}), which makes this more approachable.
\item Combining the previous two examples, in rank $2$ we could consider pairs of automorphisms that are both of finite order, hence giving us a torus orbit. This is more of a long term project, as a good description of weight modules for such TGWAs does not exist as of yet.
\end{itemize}

Additionally, our construction should have applications to categorification problems. The tower of group algebras of the symmetric groups gives a categorification of the Hopf algebra of symmetric functions (see \cite{Gei1977}). More generally, towers of algebras in the sense of \cite{BL2009} categorify dual pairs of graded Hopf algebras (which can be combined into the categorification of the Heisenberg double as done in \cite{SY2015}). It is reasonable to expect that the algebras categorified by our towers of TGWAs should be of interest.

\subsection*{Acknowledgments}
The first author gratefully acknowledges support from Simons Collaboration Grant for Mathematicians, award number 637600. The second author was supported in this research by a Grant-In-Aid of Research and a Summer Faculty Fellowship from Indiana University Northwest.

\section{Twisted generalized Weyl algebras}\label{sec:tgwa}

In this section we recall some basic definitions for twisted generalized Weyl algebras, following \cite{MT1999} where they were first defined. We also state some previous results about properties of TGWAs that we will need in what follows.

\begin{Definition}\label{def:tgwa}
Let $\K$ be a commutative ring, $R$ an associative unital $\K$-algebra, $n$ a positive integer, $\si=(\sigma_1^{1/2},\sigma_2^{1/2},\ldots,\sigma_n^{1/2})\in \Aut_\K(R)^n$ an $n$-tuple of commuting automorphisms of $R$, $t=(t_1,t_2,\ldots,t_n)\in Z(R)^n$ an $n$-tuple of central regular elements of $R$. The \emph{twisted generalized Weyl construction (TGWC) of rank $n$}, denoted $\tilde{\CA}=\tilde{\CA}(R,\si,t)$, is the associative algebra obtained from $R$ by adjoining $2n$ new generators
$X_1^\pm,X_2^\pm,\ldots,X_n^\pm$ that are not required to commute with each other, nor with the elements of $R$, but are subject to the following relations for all $i,j=1,2,\ldots,n$:
\begin{equation}\label{eq:TGWA-Rels}
X_i^\pm  r = \si_i^{\pm 1}(r) X_i^\pm,\qquad 
X_i^\pm X_i^\mp = \si_i^{\pm 1/2}(t_i), \qquad 
[X_i^\pm, X_j^\mp ] =0 \quad \text{if $i\neq j$}.
\end{equation}  
The \emph{twisted generalized Weyl algebra (TGWA)}, denoted $\CA(R,\si,t)$, is defined as the quotient $\tilde{\CA}/\CI$, where $\CI$ is the two sided ideal of $R$-torsion elements:
\begin{equation}
\CI=\{a\in \tilde{\CA}\mid \exists r\in R_{\mathrm{reg}}\cap Z(R): ra=0\}
\end{equation}
where $R_\mathrm{reg}$ is the set of regular elements in $R$ and $Z(R)$ is the center. 
\end{Definition}
Our definition of $\CI$ is different but equivalent to the original one from \cite{MT1999}. The equivalence is proved in \cite[Thm.~3.11]{H2018}, although there is a typo in that paper, with $Z(R)$ missing from the definition. The following property is useful when constructing homomorphisms from TGWAs.

\begin{Lemma}[{\cite[Cor.~3.12]{H2018}}] \label{lem:TGWAind}
If $B$ is any $\K$-algebra and $\tilde{\varphi}:\tilde{\CA}(R,\sigma,t)\to B$ is a $\K$-algebra homomorphism such that $\varphi(r)$ is regular in $B$ for every regular element $r$ of $R$, then $\tilde{\varphi}$ induces a $\K$-algebra homomorphism $\varphi:\CA(R,\si,t)\to B$.
\end{Lemma}

\begin{Definition}
$\CA(R,\si,t)$ is \emph{consistent} if 
the canonical map $R\to\CA(R,\si,t)$, $r\mapsto r1$, is injective.
\end{Definition}
This definition is important for us because consistency guarantees that the relations do not make the TGWA into the trivial algebra.
\begin{Theorem}[\cite{FutHar2012}]
$\CA(R,\si,t)$ is consistent if and only if
\begin{gather}
\si_i^{1/2}(t_j)\si_j^{1/2}(t_i)=
\si_i^{-1/2}(t_j)\si_j^{-1/2}(t_i)\quad\forall i\neq j \label{eq:cons1} \\
\si_i^{1/2}\si_j^{1/2}(t_k)
\si_i^{-1/2}\si_j^{-1/2}(t_k)=
\si_i^{1/2}\si_j^{-1/2}(t_k)
\si_i^{-1/2}\si_j^{1/2}(t_k) \quad\forall i\neq j\neq k\neq i. \label{eq:cons2}
\end{gather}
\end{Theorem}

\section{Towers of twisted generalized Weyl algebras}\label{section:towers}
In the rest of the paper we assume that $R$ is an integral domain\footnote{More generally we may work with a noncommutative domain $R$ having an anti-automorphism $\ast$ which is the identity on the center.}. Fix $\sigma=(\sigma_1^{1/2}, \ldots, \sigma_n^{1/2})\in \Aut(R)^n$, an $n$-tuple of commuting automorphism of $R$. For each $t=(t_1,\ldots, t_n)\in (R\setminus\{0\})^n$ such that $\CA(R,\sigma,t)$ is consistent, we define the TGWA $A(t):=\CA(R,\sigma,t)$.

\begin{Definition}
We define the set of all solutions to the consistency equations 
$$\Omega=\Omega(R,\si)=\{t\in (R\setminus\{0\})^n~|~\CA(R,\si,t) \text{ is consistent }\}.$$
\end{Definition}

\begin{Lemma} If $t,t'\in\Omega$ then $t\cdot t'=(t_1t_1',t_2t_2',\ldots,t_nt_n')\in \Omega$.
\end{Lemma}

\begin{proof}
 Straightforward from \eqref{eq:cons1}-\eqref{eq:cons2}, using that $\sigma_i^{1/2}$ are multiplicative.
\end{proof}

Let $A(t)\otimes_R A(t')$ denote the tensor product of $A(t)$ and $A(t')$ viewed as left modules over the commutative ring $R$. Explicitly

\begin{equation}\label{eq:Rtensor}
A(t)\otimes_R A(t') := \big(A(t)\otimes_\Z A(t')\big)/\big((ra)\otimes b - a\otimes (rb)\mid r\in R, a\in A(t), b\in A(t')\big)
\end{equation}

\begin{Theorem}\label{thm:coprod}
\begin{enumerate}[{\rm (a)}]
\item For any two solutions to the consistency equations $t, t'\in R^n$, there is a homomorphism of $R$-rings
\begin{equation}\label{eq:tgwa-coprod}
\Delta_{t,t'}:A(tt')\to A(t)\otimes_R A(t')
\end{equation}
which is uniquely determined by
$$ r\mapsto r\otimes 1=1\otimes r\qquad \forall r\in R$$
$$ X_i^{\pm}(tt')\mapsto X_i^{\pm}(t)\otimes X_i^{\pm}(t')\qquad i=1,\ldots,n $$ 
where $X_i^{\pm}(t)$ denotes the generators in $A(t)$.
\item For any three solutions to the consistency equations $t, t', t''\in R^n$, the following coassociative law holds:
\begin{equation}
\big(\Delta_{t,t'}\otimes \Id_{A(t'')}\big)\circ\Delta_{tt',t''} = 
\big(\Id_{A(t)}\otimes \Delta_{t,t't''}\big)\circ\Delta_{t,t't''}
\end{equation}
\item For any two solutions to the consistency equations $t, t'\in R^n$, the following cocommutative law holds:
\begin{equation}
P\circ \Delta_{t,t'} = \Delta_{t',t},
\end{equation}
where $P(x\otimes y)=y\otimes x$.
\end{enumerate}
\end{Theorem}

\begin{proof}
(a) Let $F$ be the free $R$-ring on $\{X_i^\pm\}_{i=1}^n$. By the universal property of free $R$-rings there exists a homomorphism
\[
\Delta:F\to A(t)\otimes_R A(t')
\]
uniquely determined by the conditions
\[
\Delta(X_i^\pm)=X_i^\pm\otimes X_i^\pm,\quad \Delta(r)=r\otimes 1=1\otimes r, r\in R.
\]
We have
\begin{align*}
\Delta\big(X_i^\pm r-\si_i^{\pm 1}(r) X_i^\pm\big) &= 
\Delta(X_i^\pm)\Delta(r)-\Delta\big(\si_i^{\pm 1}(r)\big) \Delta(X_i^\pm)\\
&= (X_i^\pm\otimes X_i^\pm) (r\otimes 1) - \big(\si_i^{\pm 1}(r)\otimes 1\big)(X_i^\pm\otimes X_i^\pm) \\
&= (X_i^\pm r - \si_i^{\pm 1}(r)X_i^\pm)\otimes X_i^\pm
= 0
\end{align*}

\begin{align*}
\Delta(X_i^\pm X_i^\mp - \si^{\pm 1/2}(t_it_i')\big)&=
\Delta(X_i^\pm)\Delta(X_i^\mp)-\Delta\big(\si_i^{\pm 1/2}(t_it_i')\big)\\
&= X_i^\pm X_i^\mp \otimes X_i^\pm X_i^\mp -\si_i^{\pm 1/2}(t_i)\si_i^{\pm 1/2}(t_i')\otimes 1 \\
&= \si_i^{\pm 1/2}(t_i)\otimes \si_i^{\pm 1/2}(t'_i) - \si_i^{\pm 1/2}(t_i)\otimes \si_i^{\pm 1/2}(t_i') 
= 0.
\end{align*}

Similarly one checks that 
\[ \Delta(X_i^+X_j^--X_j^-X_i^+)=0,\quad i\neq j. \]
This proves that $\Delta$ induces a homomorphism
\[
\tilde{\Delta}_{t,t'}:\tilde{A}(tt')\to A(t)\otimes_R A(t')
\]
where $\tilde{A}(tt')$ is the TGWC $\tilde{\CA}(R,\si,tt')$.
By Lemma \ref{lem:TGWAind}, it remains to show that $\tilde{\Delta}_{t,t'}(r)$ is regular for all nonzero $r\in R$. Since $A(t)$ and $A(t')$ are consistent TGWAs, they are torsion-free $R$-modules.
Localizing at the set of nonzero elements of $R$, $A(t)$ and $A(t')$ become vector spaces over the field of fractions $F$ of $R$. Thus $A(t)\otimes_R A(t')$ can be embedded into the $F$-vectorspace $(F\otimes A(t)\big)\otimes_F \big(F\otimes A(t')\big)$, on which any $r\in R$ obviously acts injectively.

(b) Since the maps involved are homomorphisms, it suffices to check that equality holds when each side is evaluated on the generators. Indeed, when evaluated at $X_i^\pm(tt't'')$ both sides become $X_i^\pm(t)\otimes X_i^\pm(t')\otimes X_i^\pm(t'')$. And at $r\in R$ both sides equal $r\otimes 1\otimes 1$ ($=1\otimes r\otimes 1=1\otimes 1\otimes r$).

(c) The argument is similar to part (b).
\end{proof}

\begin{Example}\label{exa:sl2}
Let $n=1$, $R=\C[z]$, $\si(z)=z-1$, then for any $k\in \C$ we have (see \cite{Bavula1992}) an isomorphism $A(z-k)\cong A_1(\C)= \C\langle x,\partial\rangle /(\partial x-x\partial-1)$ given by
$$X^+\mapsto x,\qquad  X^-\mapsto \partial ,\qquad z\mapsto 
\frac{1}{2}(x\partial+\partial x)+k.$$
For any $k,l\in\C$ such that $k-l\in\Z$, we have (see \cite{Bavula1992}) also an isomorphism $U(\Fsl_2)/(C-\la)\cong A((z-k)(z-l))$ given by
$$ e\mapsto \sqrt{-1}X^+,\qquad f\mapsto \sqrt{-1}X^-,\qquad h\mapsto 2z-(k+l)$$
where $C=ef+fe+h^2/2$ is the Casimir operator and $\la=\frac{1}{2}\big((k-l)^2+1\big)$.
Then \eqref{eq:tgwa-coprod} gives an algebra map 
\begin{equation}\label{eq:sl2toweyl}U(\Fsl_2)/(C-\la)\cong A((z-k)(z-l))\to A(z-k)\otimes_R A(z-l)\cong A_1(\C,x)\otimes_{R} A_1(\C,y)\end{equation}
where $A_1(\C,x)$ has generators $x,\partial_x$, and $A_1(\C,y)$ has generators $y,\partial_y$ and $R=\C[z]$ acts on $A_1(\C,x)$ (resp. $A_1(\C,y)$) by $z\mapsto\frac{1}{2}(x\partial_x+\partial_x x)+k$ (resp. $z\mapsto\frac{1}{2}(y\partial_y+\partial_y y)+l$). 

The map \eqref{eq:sl2toweyl} is given on the generators by
\begin{equation}\label{eq:sl2act} e\mapsto \sqrt{-1}x\otimes y,\qquad f\mapsto \sqrt{-1}\partial_x\otimes \partial_y,\qquad h\mapsto \frac{1}{2}(x\partial_x+\partial_x x) \otimes 1+1\otimes\frac{1}{2}(y\partial_y+\partial_y y). \end{equation}
\end{Example}


\begin{Definition}
If $A$ is an $R$-ring, an $A$-module $M$ is called a  \emph{weight module} if
\[
M=\bigoplus_{\Fm\in\MaxSpec(R)} M_\Fm,\qquad
M_\Fm=\{v\in M\mid \Fm v=0\}.
\]
For a weight module $M$, we define the \emph{support} of $M$ to be
$$ \Supp(M):=\{\Fm\in\MaxSpec(R)~|~M_\Fm\neq 0\}.$$
For $t\in\Omega$, we let $A(t)\text{-wmod}$ be the category of weight modules $M$ for $A(t)$ such that $M_\Fm$ is a finite dimensional $R/\Fm$ vector space for all $\Fm\in\MaxSpec(R)$.
\end{Definition}


\begin{Remark}Since the automorphisms $\{\si_i\}_{i=1}^n$ commute, the group $\Z^n$ acts on $\MaxSpec(R)$ by $(g_1,\ldots,g_n).\Fm = \si_1^{g_1}\cdots\si_n^{g_n}(\Fm)$.
For each $\Z^n$-orbit $\mathcal{O}$ in $\MaxSpec(R)$ there is a corresponding full subcategory of weight modules whose support is contained in $\mathcal{O}$. We denote this subcategory by $A(t)\text{-wmod}_\mathcal{O}$.
\end{Remark}

We have the following result about the category of weight modules, which is mostly well-known. 

\begin{Theorem}\label{thm:wmodO}
Let $A(t)=\CA(R,\si,t)$ be a consistent TGWA, 
\begin{enumerate}[{\rm (a)}]
\item ${\displaystyle A(t)\text{-wmod}\simeq \prod_{\mathcal{O}\in\MaxSpec(R)/\Z^n}A(t)\text{-wmod}_\mathcal{O}}$.
\item $A(t)\text{-wmod}_\mathcal{O} \simeq A_{\mathcal{O}}(\bar{t})\text{-wmod}_\mathcal{O}$ where $A_{\mathcal{O}}(\bar{t})$ is the TGWA $\CA(R_\CO,\bar\si,\bar t)$ where $R_\CO$ is the localization $S_\CO^{-1}R$, $S_\CO=R\setminus\cup_{\Fm\in\CO}\Fm$, $\bar\si_i$ are the induced automorphisms and $\bar t_i=1^{-1} t_i$.
\item If $\al$ is an automorphism of $R$ which commutes with all $\si_i$ then
$A(t)\text{-wmod}_\mathcal{O}\simeq \CA(R,\si,\al(t))\text{-wmod}_{\alpha(\mathcal{O})}$.
\item If, for all $i$, $\overline{\si_i^{1/2}(t_i)}$ is invertible in $S_\CO^{-1}R$ and $\CO$ is a torsion-free orbit, then the category $A(t)\text{-wmod}_\mathcal{O}$ is equivalent to the category of finite-dimensional vector spaces (i.e. it is semisimple with a unique simple object up to isomorphism). Moreover, the unique simple object is isomorphic, as an $R$-module, to $\bigoplus_{\Fm\in\CO}R/\Fm$.
\end{enumerate}
\end{Theorem}
\begin{proof}
(a) This is well-known, see e.g. \cite{MT1999}.

(b) By the localization results in \cite{FutHar2012} we have $S_\CO^{-1}\CA(R,\si,t)\cong \CA(S_\CO^{-1}R,\bar \si, \bar t)$. Using that $\MaxSpec(R_\CO)=\CO$, it is easy to check that the functor $A(t)\text{-wmod}_\CO\to A_\CO(\bar{t})\text{-wmod}_\CO$ given by $M\mapsto S_\CO^{-1}M$ is an equivalence of categories.

(c) By \cite{FutHar2012}, the isomorphism $\al:R_{\CO}\to R_{\al(\CO)}$ lifts to an isomorphism $\CA(R_\CO,\bar \si, \bar t)\to \CA(R_{\al(\CO)},\bar\si,\al(\bar t))$.

(d) This is proved in \cite{MPT2003}, notice that the $\si_i^{1/2}$ are appearing here because of our choice of using the symmetric version of the relations and consistency equations. 
\end{proof}

\begin{Lemma}\label{lemma:weightmod}
If $M$ is a weight module for $A(t)$ and $M'$ is a weight module for $A(t')$, then $M\otimes_R M'$ (defined as in \eqref{eq:Rtensor} is a weight module for $A(t)\otimes_R A(t')$ via the action $(a\otimes a')\cdot (v\otimes v')=av\otimes a'v'$. In particular we have $(M\otimes_R M')_\Fm=M_\Fm\otimes_R M'_\Fm$ for all $\Fm\in\MaxSpec(R)$ and $\Supp(M\otimes_R M')=\Supp(M)\cap \Supp(M')$. 
\end{Lemma}
\begin{proof}In order to show that the action is well defined, we need to check that if $c\otimes d\in A(t)\otimes_R A(t')$ is in the two sided ideal generated by all elements of the form $(ra)\otimes b - a\otimes (rb)$, $r\in R$, then it acts as zero on $M\otimes_R M'$. Let $c\otimes d=\sum_i (c_i\otimes c_i')((ra_i)\otimes b_i - a_i\otimes (rb_i))(d_i\otimes d_i')$, then
\begin{align*}
(c\otimes d)\cdot (v\otimes v') &= \sum_i (c_i\otimes c_i')((ra_i)\otimes b_i - a_i\otimes (rb_i))(d_i\otimes d_i')(v\otimes v') \\
&=  \sum_i (c_i\otimes c_i')((ra_i)\otimes b_i - a_i\otimes (rb_i))(d_iv\otimes d_i'v') \\
&= \sum_i (c_i\otimes c_i')(ra_id_iv\otimes b_id_i'v' - a_id_iv\otimes rb_id_i'v') \\
&= \sum_i(c_i\otimes c_i')(0) \\
&= 0.
\end{align*}
Now suppose $\Fm\neq \Fm'$ and let $v\otimes v'\in M_\Fm\otimes_R M'_{\Fm'}$. Let $a\in\Fm'\setminus \Fm$, then there exists $b\in R$ such that $ba=1+m$, with $m\in\Fm$. We have then
\begin{align*}
(ba\otimes 1)(v\otimes v')&=(1\otimes ba)(v\otimes v') \\
(ba)v\otimes v' &= v\otimes (ba)v' \\
(1+m)v\otimes v'&= v\otimes b(av') \\
v\otimes v' + (mv)\otimes v' &= 0 \\
v\otimes v' +0 &= 0 \\
v\otimes v'=0.
\end{align*}
On the other hand, if $v\otimes v'\in M_\Fm\otimes_R M'_\Fm$, then clearly for all $r\in\Fm$ we have $(r\otimes 1)(v\otimes v')=(1\otimes r)(v\otimes v')=0$.
It follows that 
\begin{align*}
M\otimes_R M'&=\left(\bigoplus_\Fm M_\Fm\right)\otimes_R \left(\bigoplus_{\Fm'} M'_{\Fm'}\right)\\
&= \bigoplus_{\Fm,\Fm'}M_\Fm\otimes_RM'_{\Fm'} \\
&= \bigoplus_{\Fm} M_\Fm\otimes_R M'_{\Fm}
\end{align*}
is indeed the weight space decomposition.
\end{proof}
We let $K_0(A(t)\text{-wmod})$ be the \emph{Grothendieck group} of the category $A(t)\text{-wmod}$. More precisely, we define
$$ K_0(A(t)\text{-wmod}):=\Z \langle [M]~|~M\in A(t)\text{-wmod}\rangle /\sim$$
where we quotient by the abelian subgroup generated by $[M'']-[M]-[M']$ for all short exact sequences $0
\to M \to M''\to M'\to 0$ in $A(t)\text{-wmod}$.

Analogously, we define $K_0^{\text{split}}(A(t)\text{-wmod})$ to be the \emph{split Grothendieck group} of the category $A(t)\text{-wmod}$, that is
$$ K_0^{\text{split}}(A(t)\text{-wmod}):=\Z \langle [M]~|~M\in A(t)\text{-wmod}\rangle /\sim$$
here we quotient by the abelian subgroup generated by $[M\oplus M']-[M]-[M']$ for all $M,M'\in A(t)\text{-wmod}$.




\begin{Proposition}
The map \eqref{eq:tgwa-coprod} induces maps
\begin{equation}\label{eq:k0m}K_0(A(t)\text{-wmod})\otimes_\Z K_0(A(t')\text{-wmod})\to K_0(A(tt')\text{-wmod})\end{equation}
\begin{equation}\label{eq:k0m-sp}K_0^{\text{split}}(A(t)\text{-wmod})\otimes_\Z K_0^{\text{split}}(A(t')\text{-wmod})\to K_0^{\text{split}}(A(tt')\text{-wmod})\end{equation}
both given by
$$ [M]\otimes [M']\mapsto [M\otimes_{R} M']$$
where $M\otimes_{R} M'$ is seen as a module for $A(tt')$ by restriction under the map \eqref{eq:tgwa-coprod}. (By Lemma \ref{lemma:weightmod} the weight spaces of $M\otimes_{R} M'$ are finite dimensional.)
\end{Proposition}

\begin{proof}
That the map is well defined on the split Grothendieck groups is clear because 
$$(M_1\oplus M_2)\otimes_R M'=(M_1\otimes_RM')\oplus (M_2\otimes_R M')$$ for all $M_1,M_2\in A(t)\text{-wmod}$ and $M'\in A(t')\text{-wmod}$.

For the case of the Grothendieck groups, we need to show that if $0\to M\stackrel{i}{\to} S\stackrel{p}\to N\to 0$ is a short exact sequence in $A(t)\text{-wmod}$, then, for all $M'\in A(t')\text{-wmod}$, 
\begin{equation}\label{eq:ses}
0\to M\otimes_R M'\stackrel{i\otimes \Id_{M'}}{\longrightarrow}S\otimes_R M'\stackrel{p\otimes \Id_{M'}}{\longrightarrow}N\otimes_R M'\to 0
\end{equation}
is a short exact sequence in $A(tt')\text{-wmod}$.

Suppose  then that $0\to M\stackrel{i}{\to} S\stackrel{p}\to N\to 0$ is a short exact sequence of weight modules for $A(t)$, in particular $i$ and $p$ are maps of $R$-modules, so for all $\Fm\in\MaxSpec(R)$ we can define $i_\Fm:M_\Fm\to S_\Fm$, and $p_\Fm:S_\Fm\to N_\Fm$. We then have that, for all $\Fm\in\MaxSpec(R)$, 
$$0\to M_\Fm\stackrel{i_\Fm}{\to} S_\Fm\stackrel{p_\Fm}\to N_\Fm\to 0$$
is a short exact sequence of $R$-modules. Since all the modules involved in the sequence are annihilated by $\Fm$, this is also a short exact sequence of $R/\Fm$-modules. Now consider the sequence
\begin{equation}\label{eq:sesmm}
0\to M_\Fm \otimes_R M'\stackrel{i_\Fm\otimes \Id_{M'}}{\longrightarrow}S_\Fm\otimes_R M'\stackrel{p_\Fm\otimes \Id_{M'}}{\longrightarrow}N_\Fm\otimes_R M'\to 0
\end{equation}
which by Lemma \ref{lemma:weightmod} is the same as the following sequence of $R$-modules
\begin{equation}\label{eq:sesmm1}0\to M_\Fm \otimes_R M'_\Fm \stackrel{i_\Fm\otimes \Id_{M'}}{\longrightarrow}S_\Fm\otimes_R M'_\Fm\stackrel{p_\Fm\otimes \Id_{M'}}{\longrightarrow}N_\Fm\otimes_R M'_\Fm\to 0.\end{equation}
Again, we can also consider \eqref{eq:sesmm1} as a sequence of $R/\Fm$-modules, and, since $R/\Fm$ is a field, $M'_\Fm$ is free over $R/\Fm$, hence the sequence \eqref{eq:sesmm1}, which is the same as \eqref{eq:sesmm}, is exact.
Since \eqref{eq:sesmm} is a short exact sequence of $R/\Fm$-modules and $R$-modules, by taking the direct sum over all $\Fm$, we get that \eqref{eq:ses} is a short exact sequence of $R$-modules. But the maps in \eqref{eq:ses} are actually maps of $A(t)\otimes_R A(t')$-modules and maps of $A(tt')$-modules, so the statement is proved.
\end{proof}
\begin{Example}
Using the map from Example \ref{exa:sl2}, we get that if $M_1$, $M_2$ are weight modules for the Weyl algebra $A_1(\C)$, and we fix two identifications $A(z-\ell)\cong A_1(\C)\cong A(z-k)$, then there is a canonical structure of $U(\Fsl_2)$-module on $M_1\otimes_{\C[z]} M_2$. 
We take $k-l\in\Z_{\geq 0}$, $A(z-k)\simeq A_1(\C,x)$, $A(z-l)\simeq A_1(\C,y)$. We let $M_x^-\simeq  A_1(\C,x)/A_1(\C,x)x$, $M_y^+\simeq A_1(\C,y)/A_1(\C,y)\partial_y$, $M_y^-\simeq A_1(\C,y)/A_1(\C,y)y$. 

Then $M_x^-= \bigoplus_{s\geq 0}\C \partial_x^s$ is a simple weight module, with $\C\partial_x^s=(M_x^-)_{(z-(k-1/2-s))}$. Analogously, 
$$M_y^-=\bigoplus_{s\geq 0}(M_y^-)_{(z-(l-1/2-s))}=\bigoplus_{s\geq 0}\C\partial_y^s,\qquad M_y^+=\bigoplus_{s\geq 0}(M_y^+)_{(z-(l+1/2+s))}=\bigoplus_{s\geq 0}\C y^s.$$

Then, we have the weight space decomposition 
$$M_x^-\otimes_{\C[z]}M_y^+=\bigoplus_{s=0}^{k-l-1}\C \partial_x^s\otimes y^{k-l-1-s}$$
which is nonzero if $k-l>0$. When $k-l>0$, let $v_s=(\sqrt{-1})^s \partial_x^s\otimes y^{k-l-1-s}$. We have that by \eqref{eq:sl2act}, the action of the generators of $\Fsl_2$ is given by
\begin{align*} e\cdot v_s&= \sqrt{-1}(x\otimes y)(\sqrt{-1})^s \partial_x^s\otimes y^{k-l-1-s}\\
&=(\sqrt{-1})^{s+1}x\partial_x^s\otimes yy^{k-l-1-s} \\
&=(-1)(\sqrt{-1})^{s-1} (-s)\partial_x^{s-1}\otimes y^{k-l-1-(s-1)}\\
&=s v_{s-1}
\end{align*}
\begin{align*} f\cdot v_s&= \sqrt{-1}(\partial_x\otimes \partial_y)(\sqrt{-1})^s \partial_x^s\otimes y^{k-l-1-s}\\
&=(\sqrt{-1})^{s+1}\partial_x\partial_x^s\otimes \partial_yy^{k-l-1-s} \\
&=(\sqrt{-1})^{s+1}\partial_x^{s+1}\otimes (k-l-1-s)y^{k-l-1-(s+1)}\\
&=(k-l-1-s) v_{s+1}
\end{align*}
\begin{align*}
h\cdot v_s&= \left(\frac{1}{2}(x\partial_x+\partial_x x) \otimes 1+1\otimes\frac{1}{2}(y\partial_y+\partial_y y)\right)(\sqrt{-1})^s \partial_x^s\otimes y^{k-l-1-s}\\
&= \frac{1}{2}(\sqrt{-1})^s(x\partial_x+\partial_x x)\partial_x^s\otimes  y^{k-l-1-s}+\frac{1}{2}(\sqrt{-1})^s \partial_x^s\otimes (y\partial_y+\partial_y y)y^{k-l-1-s} \\
&= \frac{1}{2}(\sqrt{-1})^s\left((-2s-1)\partial_x^s\otimes y^{k-l-1-s}+(2(k-l-1-s)+1)\partial_x^s\otimes y^{k-l-1-s}\right)\\
&= (k-l-1-2s)(\sqrt{-1})^s\partial_x^s\otimes y^{k-l-1-s}\\
&=(k-l-1-2s)v_s.
\end{align*}
Hence $M_x^-\otimes_{\C[z]}M_y^+$ is isomorphic to the irreducible finite dimensional representation of $\Fsl_2$ with highest weight $k-l-1$.
With a similar computation, it can be checked that for $k-l\geq 0$, $M_x^-\otimes_{\C[z]}M_y^-=\bigoplus_{s\geq 0}\C \partial_x^{k-l+s}\otimes\partial_y^s$ is isomorphic to the irreducible Verma module with highest weight $l-k-1$. In fact, we can also obtain nonirreducible Verma modules as tensor products in a similar way, see Proposition \ref{prp:tensor-decomposition-of-indecomposables}.
\end{Example}
\begin{Remark}
It would be tempting to restrict ourselves to the subcategory of $A(t)\text{-wmod}$ consisting of modules of finite length, unfortunately the tensor product of two finite length modules need not be finite length in general, as the next example shows. It is however true that, for many choices of $(R,\sigma)$, finite length modules will be closed under taking tensor products.
\end{Remark}

\begin{Example}
We provide an example of $(R,\si)$ in which the tensor product of two simple modules need not have finite length.

Let $R$ be the algebra of entire functions in the complex plane. Put $\boldsymbol{i}=\sqrt{-1}$. Let $n=2$ and define $\si_1\big(f(z)\big)=f(z+1)$ and $\si_2\big(f(z)\big)=f(z+\boldsymbol{i})$.
Pick an entire function $\zeta:\C\to\C$ with zero set equal to $\Z$. For example one can take $\zeta(z)=\exp(2\pi\boldsymbol{i}z)-1$.
Let
\begin{equation}
t_1(z)=\zeta\big(\frac{z+1/2}{1+\boldsymbol{i}}\big),\quad
t_2(z)=\zeta\big(\frac{z-\boldsymbol{i}/2}{1+\boldsymbol{i}}\big).
\end{equation}
Then $t=(t_1,t_2)$ solve the consistency equations \eqref{eq:cons1}.
Furthermore, let
\begin{equation}
s_1(z)=\zeta\big(\frac{z-1/2}{1+\boldsymbol{i}}\big),\quad
s_2(z)=\zeta\big(\frac{z+\boldsymbol{i}/2}{1+\boldsymbol{i}}\big).
\end{equation}
Then $s=(s_1,s_2)$ is another solution to \eqref{eq:cons1}.

For $(a,b)\in\Z^2$, let $\Fm_{(a,b)}$ be the maximal ideal $\big(z-(a+b\boldsymbol{i})\big)$ of $R$, and consider the integral orbit $\mathcal{O}=\{\Fm_{(a,b)}=\big(z-(a+b\boldsymbol{i})\big)\mid (a,b)\in\Z^2\}$ which is torsion-free.

Since $R$ is a domain and $\Z^2$ acts faithfully on $R$ (since $\si_1^k\si_2^l(z)=z+(k+\boldsymbol{i}l)$), it follows by \cite[Thm.~5.1]{HarOin} that $R$ is maximal commutative in any $A(f)$ for any solution $f=(f_1,f_2)$, (with $f_i\neq 0$), to the consistency equations \eqref{eq:cons1}.

Therefore, according to the main results of \cite{H2006} (more clearly explained in  \cite[Sec.~3.5]{H2018}), there is a bijection between the set of isomorphism classes of simple weight $A(f)$-modules and connected components of $\Specm(R)$, where connectedness $\sim$ is defined to be the smallest equivalence relation such that $\si_i^{-1/2}(\Fm)\sim\si_i^{1/2}(\Fm)$ if $f_i\notin\Fm$ for all $i$.
Note that different orbits are always disconnected. 

The orbit $\mathcal{O}$ under consideration here has two connected components with respect to $t$, and two with respect to $s$:
\[
\mathcal{O}=\mathcal{O}^+_t\sqcup \mathcal{O}^-_t,\qquad \mathcal{O}^+_t=\{\Fm_{a,b}\mid a< b\},\quad \mathcal{O}^-_t=\{\Fm_{a,b}\mid a\ge b\}
\]
\[
\mathcal{O}=\mathcal{O}^+_s\sqcup \mathcal{O}^-_s,\qquad \mathcal{O}^+_s=\{\Fm_{a,b}\mid a\le b\},\quad \mathcal{O}^-_s=\{\Fm_{a,b}\mid a> b\}
\]
On the other hand, $\mathcal{O}$ has infinitely many connected components with respect to $ts$:
\[
\mathcal{O}=\mathcal{O}^+_{ts}\sqcup\big(\bigsqcup_{a\in\Z}\mathcal{O}_{ts}^a\big)\cup \mathcal{O}^-_{ts},\qquad \mathcal{O}^+_{ts}=\{\Fm_{a,b}\mid a<b\},\quad \mathcal{O}^-_{ts}=\{\Fm_{a,b}\mid a> b\},
\quad \mathcal{O}^a_{ts}=\{\Fm_{a,a}\}.
\]
Thus $A(t)\textrm{-wmod}_{\mathcal{O}}$ has exactly two simple weight modules $M^\pm$ and similarly $A(s)\textrm{-wmod}_{\mathcal{O}}$ has two simple weight modules $N^\pm$.
But the $A(ts)$-module $M^-\otimes N^+$ does not have finite length, as it is the direct sum of countably infinitely many one-dimensional simple modules.
\end{Example}

\begin{Definition}
Let $\Gamma\subset \Omega$ be a multiplicatively closed subset, then we can define two $\Gamma$-graded $\C$-algebras
$$\mathscr{A}(\Gamma)=\bigoplus_{t\in\Gamma} \C\otimes_{\Z}K_0(A(t)\text{-wmod})\qquad\qquad \mathscr{A}^{\text{split}}(\Gamma)=\bigoplus_{t\in\Gamma} \C\otimes_{\Z}K_0^{\text{split}}(A(t)\text{-wmod})$$
with multiplication given by the map \eqref{eq:k0m}.
\end{Definition}
\begin{Remark}
By Theorem \ref{thm:coprod}(b)(c), multiplication in $\mathscr{A}(\Gamma)$ and $\mathscr{A}^{\text{split}}(\Gamma)$ is associative and commutative. 
\end{Remark}
\begin{Proposition}\label{prop:direct-sum}We have the following direct products of $\C$-algebras:
\begin{equation}\label{eq:directsum} \mathscr{A}(\Gamma)=\prod_{\CO\in\MaxSpec(R)/\Z^n}\scA(\Gamma,\CO),\qquad \mathscr{A}^{\text{split}}(\Gamma)=\prod_{\CO\in\MaxSpec(R)/\Z^n}\scA^{\text{split}}(\Gamma,\CO) \end{equation}
where $\scA(\Gamma,\CO)=\bigoplus_{t\in\Gamma} \C\otimes_{\Z}K_0(A(t)\text{-wmod}_\CO)$ and analogously for $\scA^{\text{split}}(\Gamma,\CO)$.
\end{Proposition}
\begin{proof}
From Theorem \ref{thm:wmodO}(a), we have the direct products in \eqref{eq:directsum} as vector spaces. By Lemma \ref{lemma:weightmod}, if $M\in A(t)\text{-wmod}_\CO$ and $M'\in A(t')\text{-wmod}_{\CO'}$ with $\CO\neq \CO'$, then $M\otimes_R M'=0$ because the supports are disjoint. This proves that the factors in \eqref{eq:directsum} are orthogonal under the multiplication we defined.
\end{proof}
\begin{Remark}
For all $\CO\in\MaxSpec(R)/\Z^n$, $t\in\Gamma$, there is a canonical surjective map of abelian groups
$$ K_0^{\text{split}}(A(t)\text{-wmod}_\CO)\twoheadrightarrow K_0(A(t)\text{-wmod}_\CO), \qquad [M]\mapsto [M],$$
which induces a canonical surjective $\C$-algebra map
\begin{equation}\label{eq:canon}\mathscr{A}^{\text{split}}(\Gamma,\CO)\twoheadrightarrow \mathscr{A}(\Gamma,\CO).\end{equation}
\end{Remark}

\begin{Remark}\label{rem:unit}
Let $\Bbo=(1,\ldots,1)\in R^n$, then for any $R$ and $\sigma$, we have $\Bbo\in\Omega$ and $A(\Bbo)\simeq R\rtimes\Z[X_1^{\pm 1}, \ldots, X_n^{\pm 1}]$ (see \cite[Example 1]{MT1999}, although in that case $\si_i=1$ for all $i$). By Theorem \ref{thm:wmodO}, if $\CO$ is a torsion-free orbit, there is a unique simple object $M^{\Bbo}\in A(\Bbo)\text{-wmod}_\CO$, and $M^{\Bbo}_\Fm\simeq R/\Fm$. Given any $\Fm_0\in\MaxSpec(R)$, and $m_0\in M^{\Bbo}_{\Fm_0}$, then $M^{\Bbo}_{\sigma_1^{g_1}\cdots\sigma_n^{g_n}(\Fm)}=R (X_1^{g_1}\cdots X_n^{g_n})m_0$ for all $(g_1,\ldots,g_n)\in\Z^n$.
\end{Remark}
\begin{Proposition}\label{prop:unital}
If $\CO$ is a torsion-free orbit, $\Bbo\in \Gamma$, and $M^{\Bbo}$ is the unique simple object in $A(\Bbo)\text{-wmod}_\CO$, then $\scA(\Gamma,\CO)$ (resp. $\scA^{\text{split}}(\Gamma,\CO)$) is unital, with identity given by $[M^{\Bbo}]\in K_0(A(\Bbo)\text{-wmod}_\CO)$ (resp. $K_0^{\text{split}}(A(\Bbo)\text{-wmod}_\CO)$ ).
\end{Proposition}
\begin{proof}
We fix $\Fm_0\in\MaxSpec(R)$, and $m_0\in M^{\Bbo}_{\Fm_0}$, so that 
$$M^{\Bbo}=\bigoplus_{\Fm\in\CO}M^{\Bbo}_{\Fm}=\bigoplus_{(g_1,\ldots,g_n)\in\Z^n}R (X_1^{g_1}\cdots X_n^{g_n})m_0.$$
Notice that we also have
$$ M=\bigoplus_{\Fm\in\CO}M_{\Fm}=\bigoplus_{(g_1,\ldots,g_n)\in\Z^n}M_{\sigma_1^{g_1}\cdots\sigma_n^{g_n}(\Fm)}.$$
We define a map $\rho:M\to M\otimes_R M^{\Bbo}$ by
$$ \rho(m)=m\otimes  (X_1^{g_1}\cdots X_n^{g_n})m_0, \quad\text{ if }m\in M_{\sigma_1^{g_1}\cdots\sigma_n^{g_n}(\Fm)}.$$
We claim that $\rho$ is an isomorphism of weight modules for $A(t)$, where $M\otimes_R M^{\Bbo}$ is considered an $A(t)=A(t\cdot \Bbo)$-module via the map $\Delta_{t,1}$ from \eqref{eq:tgwa-coprod}.
First of all, $\rho$ is indeed a map of $A(t)$-modules because
$$\rho(r\cdot m)=rm\otimes (X_1^{g_1}\cdots X_n^{g_n})m_0=(r\otimes 1)m\otimes (X_1^{g_1}\cdots X_n^{g_n})m_0=\Delta_{t,1}(r)\cdot \rho(m),$$
\begin{align*}\rho(X_i^{\pm}(t)\cdot m)&=X_i^{\pm}(t)m\otimes (X_1^{g_1}\cdots X_i^{g_i\pm 1}\cdots X_n^{g_n})m_0\\
&=X_i^{\pm}(t)m\otimes X_i^{\pm 1}(X_1^{g_1}\cdots X_n^{g_n})m_0 \\
&=(X_i^{\pm}(t)\otimes X_i^{\pm 1})(m\otimes (X_1^{g_1}\cdots X_n^{g_n})m_0) \\
&=\Delta_{t,1}(X_i^{\pm}(t))\rho(m).\end{align*}
In this computation, we have denoted the generators in $A(t)$ by $X_i^{\pm}(t)$ and the generators in $A(\Bbo)$ by $X_i^{\pm 1}$.
 
Then $\rho$ is invertible, with inverse given by the map $m\otimes r(X_1^{g_1}\cdots X_n^{g_n})m_0\mapsto rm$, so it is indeed an isomorphism.
\end{proof}

\begin{Definition}
For a fixed $R$, $\sigma$, and $\CO\in\MaxSpec(R)/\Z^n$, we let 
\begin{align*}\Omega_\CO^{\times}&:=\{t\in\Omega~|~\si_i^{1/2}(t_i)\not\in \bigcup_{\Fm\in\CO}\Fm\quad\forall i=1,\ldots,n~\}\\
&=\{t\in\Omega~|~\si_i^{1/2}(t_i)\text{ is invertible in }S_\CO^{-1}R\quad\forall i=1,\ldots,n~\}\end{align*}.
\end{Definition}
Notice that if $t,t'\in\Omega_\CO^{\times}$, then $tt'\in\Omega_\CO^{\times}$.
\begin{Lemma}\label{lemma:uniquemod}
Suppose $\CO$ be a torsion-free orbit, and let $t, t'\in\Omega_\CO^\times$. Let $M^t$ (resp. $M^{t'}$) be the unique simple module in $A(t)\text{-wmod}_\CO$ (resp. $A(t')\text{-wmod}_\CO$), then $M^t\otimes_R M^{t'}\simeq M^{tt'}$ as $A(tt')$-modules, via the map $\Delta_{t,t'}$ of \eqref{eq:tgwa-coprod}, where $M^{tt'}$ is the unique simple module in $A(tt')\text{-wmod}_\CO$.
\end{Lemma}

\begin{proof}
By Lemma \ref{lemma:weightmod}, we know that $M^t\otimes_R M^{t'}$ is a weight module for $A(tt')$, and, as $R$-modules, we have $(M^t\otimes_R M^{t'})_\Fm\simeq R/\Fm\otimes_R R/\Fm\simeq R/\Fm$. Since the category $A(tt')\text{-wmod}_\CO$ is semisimple with a unique simple object, the result follows.
\end{proof}
\begin{Corollary}\label{cor:semigrp}
If $\CO$ is a torsion-free orbit and $\Gamma\subset\Omega_\CO^\times$, then $\scA(\Gamma,\CO)\simeq\scA^{\text{split}}(\Gamma,\CO)\simeq\C[\Gamma]$ where $\C[\Gamma]$ is the semigroup ring.
\end{Corollary}
\begin{proof}
By Theorem \ref{thm:wmodO}(d), if $t\in\Omega_\CO^\times$, the category $A(t)\text{-wmod}_\CO$ is semisimple with a unique simple $M^t$ up to isomorphism, hence $K_0(A(t)\text{-wmod})=K_0^{\text{split}}(A(t)\text{-wmod})=\Z [M^t]$. The result then follows from \ref{lemma:uniquemod}.
\end{proof}
\begin{Proposition}\label{prop:units}If $\CO$ is a torsion-free orbit, $t\in\Omega$ and $t^\times\in\Omega_\CO^{\times}$ then the category $A(t)\text{-wmod}_\mathcal{O}$ is equivalent to the category $A(t\cdot t^\times)\text{-wmod}_\mathcal{O}$ via the functor
\begin{equation}\label{eq:equivtens} M\mapsto M\otimes_R M^{t^\times}\end{equation}
where $M^{t^\times}$ is the unique simple object in $A(t^\times)\text{-wmod}_\CO$.
\end{Proposition}
\begin{proof}By Theorem \ref{thm:wmodO}(b), we have $A(t)\text{-wmod}_\mathcal{O} \simeq A_{\mathcal{O}}(\bar{t})\text{-wmod}_\mathcal{O}$. Now, we define a functor $F_{\bar{t}^\times}:A_{\mathcal{O}}(\bar{t})\text{-wmod}_\mathcal{O}\to A_{\mathcal{O}}(\bar{t}\bar{t}^\times)\text{-wmod}_\mathcal{O}$ by
$$ M\mapsto M\otimes_{R_\CO} M^{\bar{t}^\times}$$
where $M^{\bar{t}^\times}$ is the unique simple object.
Since, $t^\times\in\Omega^{\times}$, its image $\bar{t}^\times\in (R_\CO)^n$ is invertible. Hence we can define another functor 
$$F_{(\bar{t}^\times)^{-1}}:A_{\mathcal{O}}(\bar{t}\bar{t}^\times)\text{-wmod}_\mathcal{O}\to A_{\mathcal{O}}(\bar{t})\text{-wmod}_\mathcal{O},\qquad N\mapsto  N\otimes_{R_\CO} M^{(\bar{t}')^{-1}} $$
where $M^{(\bar{t}^\times)^{-1}}$ is the unique simple object. Then we have the compositions
\begin{align*} F_{(\bar{t}^\times)^{-1}}F_{\bar{t}^\times}(M)&=(M\otimes_{R_\CO}M^{\bar{t}})\otimes_{R_\CO}M^{(\bar{t}^\times)^{-1}} \\
&= M\otimes_{R_\CO}(M^{\bar{t}}\otimes_{R_\CO}M^{(\bar{t}^\times)^{-1}})\\
\text{ ( by Lemma \ref{lemma:uniquemod} ) }&\simeq M\otimes_{R_\CO} M^{\Bbo} \\
\text{ ( by Prop. \ref{prop:unital} ) }&\simeq M
\end{align*}
and analogously
\begin{align*} F_{\bar{t}^\times}F_{(\bar{t}^\times)^{-1}}(N)&=(N\otimes_{R_\CO}M^{(\bar{t}^\times)^{-1}})\otimes_{R_\CO}M^{\bar{t}} \\
&= N\otimes_{R_\CO}(M^{(\bar{t}^\times)^{-1}}\otimes_{R_\CO}M^{\bar{t}})\\
\text{ ( by Lemma \ref{lemma:uniquemod} ) }&\simeq N\otimes_{R_\CO} M^{\Bbo} \\
\text{ ( by Prop. \ref{prop:unital} ) } &\simeq N.
\end{align*}
It follows that $F_{\bar{t}^\times}$ and  $F_{(\bar{t}^\times)^{-1}}$ are equivalences of categories, hence, by applying again Theorem \ref{thm:wmodO}, we have a chain of equivalences
$$A(t)\text{-wmod}_\mathcal{O} \simeq A_{\mathcal{O}}(\bar{t})\text{-wmod}_\mathcal{O}\simeq A_{\mathcal{O}}(\bar{t}\bar{t}^\times)\text{-wmod}_\mathcal{O}\simeq A(tt^\times)\text{-wmod}_\mathcal{O}$$
and the composition of functors is given exactly by \eqref{eq:equivtens}.
\end{proof}

\begin{Remark}
Let $\Gamma\subset\Omega$ be a monoid (i.e. $\Bbo\in\Gamma$), let $\CO$ be a torsion-free orbit, and let $\Gamma^\times:=\Gamma\cap\Omega_\CO^\times$, then we have a short exact sequence of monoids
\begin{equation}\label{eq:semigrp-ses} \Bbo\to \Gamma^\times\to \Gamma\to \Gamma/\Gamma^\times\to \Bbo\end{equation}
which by Cor. \ref{cor:semigrp} induces the inclusions
$$\C[\Gamma^\times]\simeq \scA(\Gamma^\times,\CO)\hookrightarrow\scA(\Gamma,\CO), \qquad \C[\Gamma^\times]\simeq \scA^{\text{split}}(\Gamma^\times,\CO)\hookrightarrow\scA^{\text{split}}(\Gamma,\CO)$$
We can also define the quotient algebras
$$\scA(\Gamma/\Gamma^\times,\CO):=\scA(\Gamma,\CO)/\left([M]-[M\otimes_R M^t]~|~\forall M,~\forall t\in\Gamma^\times\right),$$
$$\scA^{\spl}(\Gamma/\Gamma^\times,\CO):=\scA^{\spl}(\Gamma,\CO)/\left([M]-[M\otimes_R M^t]~|~\forall M,~\forall t\in\Gamma^\times\right),$$
which are graded by the quotient monoid $\Gamma/\Gamma^\times$.
\end{Remark}
\begin{Proposition}\label{prop:split-monoid}If $\CO$ is a torsion-free orbit and the short exact sequence \eqref{eq:semigrp-ses} splits, i.e. there is a monoid $\Gamma'$ such that $\Gamma=\Gamma^\times\cdot\Gamma'$ and $\Gamma^\times\cap\Gamma'=\{\Bbo\}$, then we have isomorphisms of graded $\C$-algebras
$$\scA(\Gamma/\Gamma^\times,\CO)\simeq\scA(\Gamma',\CO),\qquad\qquad\scA(\Gamma,\CO)\simeq\C[\Gamma^\times]\otimes_\C\scA(\Gamma',\CO),$$
$$\scA^{\spl}(\Gamma/\Gamma^\times,\CO)\simeq\scA^{\spl}(\Gamma',\CO),\qquad\qquad\scA^{\spl}(\Gamma,\CO)\simeq\C[\Gamma^\times]\otimes_\C\scA^{\spl}(\Gamma',\CO).$$
\end{Proposition}
\begin{proof}Since \eqref{eq:semigrp-ses} splits, for all $t\in\Gamma$, we can write in a unique way $t=t^\times t'$, with $t^\times\in\Gamma^\times$, $t'\in\Gamma'$. We define a map $\alpha$ to be the composition of the obvious inclusion with the quotient map
$$\scA(\Gamma',\CO)\hookrightarrow\scA(\Gamma,\CO)\twoheadrightarrow \scA(\Gamma/\Gamma^\times,\CO)$$
which is clearly an algebra map. We define a map going the other way
$$\beta:\scA(\Gamma/\Gamma^\times,\CO)\to \scA(\Gamma',\CO)$$
as follows. Let $M\in A(t)\text{-wmod}$, for $t=t^\times t'\in\Gamma=\Gamma^\times\cdot \Gamma'$, then by Proposition \ref{prop:units}, $M\simeq M^{t^\times}\otimes M'$ for some $M'\in A(t')\text{-wmod}$, so we define
$$ \beta([M])=[M'].$$
It is clear that $\beta$ is well defined on the quotient, and that for $t\in\Gamma'$ we have $t'=t$, so for $[M]\in\scA(\Gamma',\CO)$ 
$$\beta\circ\alpha([M])=\beta([M])=[M']=[M].$$ 
We also have, for $[M]\in\scA(\Gamma/\Gamma',\CO)$, 
$$\alpha\circ\beta([M])=\alpha[M']=[M']=[M'\otimes_R M^{t^\times}]=[M].$$
Hence $\beta$ is a two sided inverse of $\alpha$ and they are both isomorphisms.

Now consider the map
$$\delta:\scA(\Gamma,\CO)\simeq\scA(\Gamma^\times,\CO)\otimes_\C\scA(\Gamma',\CO),\qquad \delta([M])=[M^{t^\times}]\otimes [M']$$
for $M\in A(t)\text{-wmod}$ and $M\simeq M^{t^\times}\otimes_R M'$.
We now show that $\delta$ is an algebra map. If $M\in A(t)\text{-wmod}$, $N\in A(u)\text{-wmod}$, then $tu=t^\times t'u^\times u'=(tu)^\times (tu)'$, and $M\simeq M^{t^\times}\otimes_R M'$, $N\simeq M^{u^\times}\otimes_R N'$. So we have
\begin{align*} M\otimes_R N &\simeq (M^{t^\times}\otimes_R M')\otimes_R ( M^{u^\times}\otimes_R N')\\
& \simeq M^{t^\times}\otimes_R M^{u^\times}\otimes _R M'\otimes_R N' \\
& \simeq M^{t^\times u^\times}\otimes _R (M'\otimes_R N') \\
& \simeq M^{(tu)^\times}\otimes_R (M'\otimes_R N')
\end{align*}
hence $(M\otimes_R N)'\simeq M'\otimes_R N'$ and
\begin{align*} \delta([M]\cdot [N])& =\delta([M\otimes_R N]) \\
&=[M^{(tu)^\times}]\otimes [(M\otimes_R N)'] \\
&=[M^{t^\times u^\times}]\otimes [M'\otimes_R N']\\
&=[M^{t^\times}\otimes_R M^{u^\times}]\otimes  [M'\otimes_R N']\\
&=([M^{t^\times}]\cdot [ M^{u^\times}])\otimes ([M']\cdot[N'])\\
&=( [M^{t^\times}]\otimes [M']) \cdot ([ M^{u^\times}]\otimes [N']) \\
&= \delta([M])\cdot \delta([N]).
\end{align*}
We can also define 
$$\epsilon:\scA(\Gamma^\times,\CO)\otimes_\C\scA(\Gamma',\CO)\to \scA(\Gamma,\CO),\qquad \epsilon([M]\otimes [N])=[M\otimes_R N]$$
and it is clear that $\epsilon$ is the inverse of $\delta$, hence they are isomorphisms. Finally, we use $\ref{cor:semigrp}$ to conclude  the proof for the $\scA$'s. The arguments for the $\scA^{\spl}$'s are identical.
\end{proof}
\begin{Proposition}\label{prop:monoid-iso}
Let $\CO\in\MaxSpec(R)/\Z^n$ be a torsion-free orbit. Let $\Gamma_1,\Gamma_2\subset\Omega$ be submonoids, and let $\gamma:\Gamma_1\to\Omega_\CO^\times$ be a monoid map such that the map
$$\phi:\Gamma_1\to\Gamma_2,\qquad \phi(t)=\gamma(t)\cdot t$$
is a monoid isomorphism. Then we have isomorphisms of $\C$-algebras (actually graded isomorphisms if we identify the grading monoids via $\phi$)
$$ \scA(\Gamma_1,\CO)\simeq \scA(\Gamma_2,\CO),\qquad \scA^{\spl}(\Gamma_1,\CO)\simeq\scA^{\spl}(\Gamma_2,\CO).$$
\end{Proposition}
\begin{proof}
We prove the statement for the $\scA$'s, the statement for the $\scA^{\spl}$ is entirely analogous. We define a map 
$$\alpha:\scA(\Gamma_1,\CO)\to\scA(\Gamma_2,\CO),\qquad [M]\mapsto [M\otimes_R M^{\gamma(t)}]$$
where $M\in A(t)\text{-wmod}$, $t\in\Gamma_1$, $M^{\gamma(t)}$ is the unique simple module in $A(\gamma(t))\text{-wmod}$, and $M\otimes_R M^{\gamma(t)}$ is an $A(\gamma(t)\cdot t)=A(\phi(t))$-module. Notice that if $M\in A(t)\text{-wmod}$, $N\in A(u)\text{-wmod}$, then
\begin{align*}
\alpha([M]\cdot [N])&=\alpha([M\otimes_R N]) \\
&= [(M\otimes_R N)\otimes_R M^{\gamma(tu)}]\\
&= [M\otimes_R N\otimes_R M^{\gamma(t)\gamma(u)}]\\
&= [M\otimes_R N\otimes_R M^{\gamma(t)}\otimes_R M^{\gamma(u)}]\\
&= [(M\otimes_R M^{\gamma(t)})\otimes_R(N\otimes_R M^{\gamma(u)})] \\
&= \alpha([M])\cdot \alpha([N])
\end{align*}
so $\alpha$ is an algebra map. We also define $\beta:\scA(\Gamma_2,\CO)\to\scA(\Gamma_1,\CO)$ as follows:
let $u\in\Gamma_2$, $N\in A(u)\text{-wmod}$, then $N\simeq N'\otimes_R M^{\gamma(\phi^{-1}(t))}$, we define
$$ \beta([N])=[N'].$$
Then for all $M\in A(t)\text{-wmod}$, $t\in\Gamma_1$, we have
$$\beta\circ\alpha([M])=\beta([M\otimes_R M^{\gamma(t)}])=[M],$$
and for all $N\in A(u)\text{-wmod}$, $u\in\Gamma_2$, 
$$\alpha\circ\beta([N])=\alpha([N'])=[N'\otimes M^{\gamma(\phi^{-1}(u))}]=[N].$$
Hence $\alpha$ and $\beta$ are inverses of each other and they are isomorphisms.
\end{proof}



%

\section{Rank one setting (Line)}\label{sec:line}
We retain all notation of Section \ref{section:towers} and we explictly describe the algebras introduced there, in the special case of $n=1$, $R=\C[z]$, and $\sigma^{1/2}(z)=z-\frac{1}{2}$.

In this case $\Omega=\C[z]\setminus \{0\}$ and for all $t\in\Omega$ we have the TGWA $A(t)$ generated by $X^+$ and $X^-$, with relations

$$X^+X^- = \sigma^{1/2}(t),\qquad X^- X^+= \sigma^{-1/2}(t),$$
$$X^+r = \sigma(r) X^+,\qquad X^- r=\sigma^{-1}(r) X^-,$$

for all $r\in\C[z]$.

We have $\MaxSpec(R)=\{(z-\la)~|~\la\in\C\}$ with $\Z$-action given by $\sigma(z-\la)=(z-\la-1)$. The orbits of this action can then be parametrized by $\C/\Z$. If $\la+\Z\in\C/\Z$, the corresponding orbit is $\mathcal{O}_{\la+\Z}=\{(z-\la+\Z)\}$, in particular we will consider $\CO_\Z=\{(z-\la)~|~\la\in\Z\}$. 
\begin{Remark}\label{rem:splits}Let $\la\in\C/\Z$, then $\CO=\CO_{\la+\Z}$ is a torsion-free orbit, and we have
\begin{align*}\Omega^\times_\CO&=\{t\in\C[z]\setminus \{0\}~|~\si^{1/2}(t)\not\in\bigcup_{\Fm\in\CO}\Fm \}\\
&= \{t=\alpha\prod_{k\in\C}(z-k)^{n_k}~|~\si^{1/2}(t)\not\in\bigcup_{s\in\Z}(z-\la+s)\}\\
&=  \{\alpha\prod_{k\in\C}(z-k)^{n_k}~|~\alpha\prod_{k\in\C}(z-k-1/2)^{n_k}\not\in\bigcup_{s\in\Z}(z-\la+s)\} \\
&= \{\alpha\prod_{k\in\C}(z-k)^{n_k}~|~n_k=0\text{ if }k+1/2\in\la+\Z\}\\
&=\{\alpha\prod_{k\not\in\la+1/2+\Z}(z-k)^{n_k}~|~\alpha\in\C^\times, ~n_k\geq 0\}.
\end{align*}
In all the products here, $n_k>0$ for only finitely many terms.

Then, the short exact sequence of monoids 
$$ 1\to \Omega_\CO^\times \to \Omega\to \Omega/\Omega_\CO^\times \to 1$$
splits, with 
$$\Omega/\Omega_\CO^\times\simeq \Omega':=\{\prod_{s\in\Z}(z-\la-1/2-s)^{n_s}\in\Omega~|~n_s\geq 0\}.$$ 
In particular, the reason this short exact sequence splits is that in a polynomial ring we always have a canonical choice of monic polynomials as representatives of each maximal ideal.
\end{Remark}
To describe $\scA(\Omega)$ and $\scA^{\spl}(\Omega)$, by Proposition \ref{prop:direct-sum}, it is enough to describe $\scA(\Omega,\CO)$ and $\scA^{\spl}(\Omega,\CO)$. Also, by Remark \ref{rem:splits} and Proposition \ref{prop:split-monoid} we have
$\scA(\Omega,\CO)\simeq \C[\Omega_\CO^\times]\otimes \scA(\Omega',\CO)$ (and similarly for $\scA^{\spl}(\Omega,\CO)$)) so it is enough to describe $\scA(\Omega',\CO)$. Since all the orbits $\CO\in\MaxSpec(\C[z])/\Z$ are isomorphic, we will only explicitly examine the case of $\CO=\CO_\Z$, the other orbits will give isomorphic algebras. For the rest of this section, we fix $\CO=\CO_\Z$ and
$$ \Omega':=\{t\in\C[z]~|~t=\prod_{s\in\Z+\frac{1}{2}}(z-s)^{n_s},~n_s\geq 0\}.$$

\subsection{The algebra $\scA(\Omega',\CO)$}

Since $1\in\Omega'$ and $\CO$ is torsion-free, by Proposition \ref{prop:unital} the algebra $\scA(\Omega',\CO)$ is unital, with identity element $[M^1]$, where $M^1\in A(1)\text{-wmod}_\CO$ is the unique simple module.

Let $t\in\Omega'$, $t\neq 1$, then $t=\prod_{s\in\Z+\frac{1}{2}}(z-s)^{n_s}$, and we consider the set of zeros of $t$,
$$ Z(t):=\left\{s\in\Z+\frac{1}{2}~|~n_s>0\right\},$$
which we order and extend with infinities on both sides
$$ \hat{Z}(t):=\{s_0=-\infty<s_1<s_2<\ldots<s_\ell<s_{\ell+1}=\infty~|~s_i\in Z(t), ~i=1,\ldots,\ell=\ell(t)\}.$$
\begin{Proposition}[{\cite{Bavula1992}}]\label{prop:simples}Up to isomorphism, the simple weight modules for $A(t)$ are 
$$\{ M^t_{s_i,s_{i+1}}~|~s_i,s_{i+1}\in\hat{Z}(t),~i=0,\ldots,\ell(t) \}$$
where $\Supp(M^t_{s_i,s_{i+1}})=\{(z-k)~|~k\in\Z,~s_i<k<s_{i+1}\}$. 

Moreover, $(M^t_{s_i,s_{i+1}})_{(z-k)}\simeq \C v_k$ is one dimensional, and the action of $A(t)$ satisfies the following
$$X^+ v_k=0,\text{ if }k+1>s_{i+1}\qquad X^+ v_k\in\C^\times v_{k+1},\text{ if }k+1<s_{i+1}$$
$$X^- v_k=0,\text{ if }k-1<s_{i}\qquad X^- v_k\in\C^\times v_{k-1},\text{ if }k-1>s_{i}.$$
\end{Proposition}
From Proposition \ref{prop:simples}, the following results follow immediately.
\begin{Corollary}Let $t\in\Omega'$, $t\neq 1$, then
$$K_0(A(t)\text{-wmod}_\CO)\simeq \bigoplus_{i=0}^{\ell(t)} \Z [M^t_{s_i,s_{i+1}}].$$
\end{Corollary}
\begin{Corollary}\label{cor:multiplication}
Let $t,u\in\Omega'\setminus\{1\}$, $s_i,s_{i+1}\in\hat{Z}(t)$, $w_j,w_{j+1}\in\hat{Z}(u)$ then
$$ M^t_{s_i,s_{i+1}}\otimes_{\C[z]} M^u_{w_j,w_{j+1}}\simeq \begin{cases} M^{tu}_{\max\{s_i,w_j\},\min\{s_{i+1},w_{j+1}\}} & \text{ if }\max\{s_i,w_j\}<\min\{s_{i+1},w_{j+1}\}, \\
0 & \text{ otherwise.}\end{cases}$$
\end{Corollary}
We can then give a description for the algebra $\scA(\Omega',\CO)$ in terms of a basis.
\begin{Theorem}As $\Omega'$-graded vector spaces, we have
$$\scA(\Omega',\CO)\simeq \C[M^1]\oplus \bigoplus_{t\in\Omega',~t\neq 1}\bigoplus_{i=0}^{\ell(t)}\C[M^t_{s_i,s_{i+1}}]$$
with multiplication in $\scA(\Omega',\CO)$ given by Corollary \ref{cor:multiplication} and the fact that $[M^1]$ is the identity.
\end{Theorem}
\begin{proof}
This follows directly from the above.
\end{proof}
\begin{Theorem}We have an isomorphism of $\Omega'$-graded algebras
$$\scA(\Omega',\CO)\simeq \C[x_s^{\pm}~|~s\in\Z+\tfrac{1}{2}]/(x_s^-x_w^+~|~ s\leq w)$$
given by 
$$ x_s^-\mapsto [M^{z-s}_{-\infty,s}],\qquad x_s^+\mapsto  [M^{z-s}_{s,\infty}].$$
\end{Theorem}
\begin{proof}
Consider the map 
$$\alpha: \C[x_s^{\pm}~|~s\in\Z+\tfrac{1}{2}]\to \scA(\Omega',\CO),\qquad x_s^-\mapsto [M^{z-s}_{-\infty,s}],\qquad x_s^+\mapsto  [M^{z-s}_{s,\infty}].$$
We show that $\alpha$ is surjective by proving that $\scA(\Omega',\CO)$ is generated as an algebra by $[M^{z-s}_{-\infty,s}]$, $[M^{z-s}_{s,\infty}]$, $s\in\Z+\frac{1}{2}$. We will proceed by induction on the degree of $t$. If $t\in\Omega'$, $\deg t=1$, then $t=z-s$, $s\in\Z +\frac{1}{2}$ and  $[M^{z-s}_{-\infty,s}]$, $[M^{z-s}_{s,\infty}]$ form a basis for $K_0(A(t)\text{-wmod}_\CO)$. Now suppose that $\deg t>1$, $t=(z-s_1)^{n_1}\cdots (z-s_\ell)^{n_\ell}$, $ n_i\geq 1$, $i=1,\ldots,\ell$. Let $\bar{t}=\frac{t}{z-s_1}$. If $n_1>1$, then $\hat{Z}(\bar{t})=\hat{Z}(t)$ and
$$ M^t_{s_i,s_{i+1}}\simeq\begin{cases}M^{\bar{t}}_{-\infty,s_{1}}\otimes_{\C[z]}M^{z-s_1}_{-\infty,s_1} & \text{ if }i=0, \\ 
M^{\bar{t}}_{s_i,s_{i+1}}\otimes_{\C[z]} M^{z-s_1}_{s_1,\infty} & \text{ if }i\geq 1.\end{cases}$$
If $n_1=1$, then $\hat{Z}(\bar{t})=\{-\infty<s_2<s_3<\ldots s_\ell<\infty\}$ and hence we have
$$ M^t_{s_i,s_{i+1}}\simeq\begin{cases}M^{\bar{t}}_{-\infty,s_2}\otimes_{\C[z]}M^{z-s_1}_{-\infty,s_1} & \text{ if }i=0, \\ 
M^{\bar{t}}_{-\infty,s_2}\otimes_{\C[z]}M^{z-s_1}_{s_1,\infty} & \text{ if }i=1, \\ 
M^{\bar{t}}_{s_i,s_{i+1}}\otimes_{\C[z]} M^{z-s_1}_{s_1,\infty} & \text{ if }i\geq 2.\end{cases}$$
The claim follows then by the inductive hypothesis, since $\deg\bar{t}<\deg t$.

By Corollary \ref{cor:multiplication}, if $s\leq w$, we have 
$$[M^{z-s}_{-\infty,s}]\cdot  [M^{z-w}_{w,\infty}]=[M^{z-s}_{-\infty,s}\otimes_{\C[z]}M^{z-w}_{w,\infty}]=0$$
hence $(x_s^-x_w^+~|~ s\leq w)\subset\ker(\alpha)$ and we get an induced map
$$ \C[x_s^{\pm}~|~s\in\Z+\tfrac{1}{2}]/(x_s^-x_w^+~|~ s\leq w)\ronto \scA(\Omega',\CO)$$
and the result will follow from the fact that this induced map is injective. Notice that a basis for $\C[x_s^{\pm}~|~s\in\Z+\tfrac{1}{2}]/(x_s^-x_w^+~|~ s\leq w)$ is given by
$$ \{ (x^+_{w_1})^{n_1}(x^+_{w_2})^{n_2}\cdots (x^+_{w_a})^{n_a}~|~w_1<w_2<\cdots<w_a, ~n_i\geq 1\}\cup $$
$$ \cup\{ (x^-_{s_1})^{m_1}(x^-_{s_2})^{m_2}\cdots (x^-_{s_b})^{m_b}~|~s_1<s_2<\cdots<s_b,~m_j\geq 1\}\cup $$
$$ \cup \{ (x^+_{w_1})^{n_1}\cdots (x^+_{w_a})^{n_a}(x^-_{s_1})^{m_1}\cdots (x^-_{s_b})^{m_b}~|~w_1<\cdots<w_a<s_1<\cdots s_b,~n_i,m_j\geq 1\}\cup\{1\}.$$
Then 
\begin{align*} (x^+_{w_1})^{n_1}(x^+_{w_2})^{n_2}\cdots (x^+_{w_a})^{n_a}&\mapsto \left[M^{\prod_{i=1}^a(z-w_i)^{n_i}}_{w_a,\infty}\right]\\
(x^-_{s_1})^{m_1}(x^-_{s_2})^{m_2}\cdots (x^-_{s_b})^{m_b}&\mapsto \left[M^{\prod_{j=1}^b(z-s_j)^{m_j}}_{-\infty,s_1}\right]\\
(x^+_{w_1})^{n_1}\cdots (x^+_{w_a})^{n_a}(x^-_{s_1})^{m_1}\cdots (x^-_{s_b})^{m_b}&\mapsto \left[M^{\prod_{i=1}^a(z-w_i)^{n_i}\prod_{j=1}^b(z-s_j)^{m_j}}_{w_a,s_1}\right]\\
1 &\mapsto [M^1]
\end{align*}
which shows that the image of a basis is linearly independent, hence the map is injective, which concludes the proof.
\end{proof}

\subsection{The algebra $\scA^{\operatorname{split}}(\Omega',\CO)$}

%

Indecomposable weight modules over a rank one generalized Weyl algebra were classified in \cite{DGO}. 
In the case of torsion-free orbit the simple modules are determined by their support. However the indecomposable modules are not, and it will be convenient for our purpose of the description of tensor products to introduce a notion of ``directed'' subsets. We only need to consider subsets of $\mathbb{R}$.

\subsubsection{Directed subsets}

A \emph{directed subset} of $\mathbb{R}$ is a subset  where some of the elements have a directionality (left/right) to their membership. Formally, a directed subset $S$ of $\R$ is a function $S:\mathbb{R}\to \{0,1,\rightarrow,\leftarrow\}$.
We write $k\in S$ if $S(k)\in\{1,\rightarrow,\leftarrow\}$ and $k\notin S$ if $S(k)=0$. We further write $k\overset{\rightarrow}{\in}S$ if $S(k)=\rightarrow$ and $k\overset{\leftarrow}{\in}S$ if $S(k)=\leftarrow$. We call $k$ a \emph{directed element of $S$} if $S(k)\in\{\rightarrow,\leftarrow\}$.
We say $k$ is an \emph{undirected element of $S$} if $S(k)=1$.

The \emph{underlying set} of $S$ is defined to be $\overset{\circ}{S}:=S^{-1}(\{1,\rightarrow,\leftarrow\})$.


Any subset of $\R$ can be thought of as a directed subset taking values in $\{0,1\}$.
The \emph{intersection} $S\cap T$ of any two directed subsets $S$ and $T$ is defined to be
\begin{equation}
(S\cap T)(k)=S(k)\cdot T(k)\quad \forall k\in\R,
\end{equation}
where $\cdot$ is the commutative binary operation on $\{0,1,\rightarrow,\leftarrow\}$ satisfying
\begin{equation}
1\cdot x=x,\quad  0\cdot x=0,\quad x\cdot x=x,\quad
\rightarrow\cdot\leftarrow=0.
\end{equation}


We will be interested in directed subsets of $\R$ whose underlying sets are open, and directed elements are half-integers.

\begin{Example}
Let $S$ be the directed subset $S$ of $\R$ given by
\[\overset{\circ}{S}=\left(\frac{1}{2},\infty\right),\quad \frac{3}{2}\overset{\rightarrow}{\in}S,\quad  \frac{5}{2}\overset{\leftarrow}{\in}S, \quad \frac{7}{2}\overset{\leftarrow}{\in}S\]
and remaining elements undirected.
Similarly let $T$ be given by 
\[\overset{\circ}{T}=\left(-\infty,\frac{9}{2}\right),\quad -\frac{1}{2}\overset{\leftarrow}{\in}T,\quad \frac{3}{2}\overset{\leftarrow}{\in}T,\quad \frac{7}{2}\overset{\leftarrow}{\in}T\]
and remaining elements undirected.
Let $U=S\cap T$. Then
\[ \overset{\circ}{U} = \left(\frac{1}{2},\frac{3}{2}\right)\cup\left(\frac{3}{2},\frac{9}{2}\right),\quad  \frac{5}{2}\overset{\leftarrow}{\in} U,\quad \frac{7}{2}\overset{\leftarrow}{\in} U
\]
and remaining elements undirected.
\end{Example}

We say $S$ and $T$ are \emph{strongly disjoint} if the underlying sets of $S$ and $T$ are disjoint.
The \emph{union} $S\cup T$ of $S$ and $T$, defined when $S$ and $T$ are strongly disjoint, is defined to be
\begin{equation}
(S\cup T)(x)=S(x)+T(y), \qquad \text{where $x+0=x=x+0$ for $x\in\{0, 1,\rightarrow,\leftarrow\}$.}
\end{equation}

Lastly, we say that a directed subset $S$ of $\R$ is \emph{connected} if $\overset{\circ}{S}$ is a connected subset of $\R$.



\subsubsection{Semi-indecomposable modules}
It turns out it is easier to describe the tensor product rule if we generalize indecomposable modules to what we call semi-indecomposable modules. The reason is that the class of indecomposable modules is not closed under the tensor product, but the wider class of semi-indecomposable modules is.

\begin{Definition}
A module is \emph{semi-indecomposable} if it is a direct sum of pairwise non-isomorphic indecomposable modules.
\end{Definition}

To describe these we need admissible directed subsets.
\begin{Definition}
Let $t\in\Omega'$ be a monic polynomial with zero set equal to $\{k_1,k_2,\ldots,k_r\}\subseteq\Z+\frac{1}{2}$.
A directed subset $S$ of $\mathbb{R}$ is \emph{$t$-admissible} if it satisfies the following conditions:
\begin{enumerate}[{\rm (i)}]
\item The underlying set of $S$ is an intersection of open intervals of the form $(k_i,\infty)$ and $(-\infty,k_i)$.
\item The directed elements of $S$ are exactly the elements of $\overset{\circ}{S}\cap \{k_1,k_2,\ldots,k_r\}$.
\end{enumerate}
The set of $t$-admissible directed subsets of $\mathbb{R}$ will be denoted by $\mathcal{P}^{\mathrm{dir}}_t(\mathbb{R})$.
\end{Definition}

We have
\begin{align}
S\cap S' & \in \mathcal{P}^{\mathrm{dir}}_{tt'}(\R) 
\quad\text{for all $(S,S') \in \mathcal{P}^{\mathrm{dir}}_{t}(\R) 
\times \mathcal{P}^{\mathrm{dir}}_{t'}(\R)$} \\
S\cup S' & \in \mathcal{P}^{\mathrm{dir}}_{t}(\R) 
\quad\text{for all strongly disjoint $S,S' \in \mathcal{P}^{\mathrm{dir}}_{t}(\R)$.}
\end{align}

The point now is that  $\mathcal{P}^{\mathrm{dir}}_t(\mathbb{R})$ precisely parametrize semi-indecomposable modules for a generalized Weyl algebra $A(t)$.
This part is a direct consequence of a special case of \cite{DGO}. But moreover, the tensor product simply corresponds to intersection. More precisely, we have the following:

\begin{Theorem}\label{thm:a-theorem-about-modules}
Let $R=\C[z]$, $\si(z)=z-1$, and
\[t=(z-k_1)^{m_1}(z-k_2)^{m_2}\cdots (z-k_r)^{m_r}
\in\C[z],\]
where $r\in\Z_{\ge 0}$, $k_i\in \Z+\frac{1}{2}$, $m_i\in\Z_{>0}$. (So $t=1$ if $r=0$.) Let $\CO=\{(z-x)\mid x\in\Z\}$ be the integral orbit in $\Specm(\C[z])$ under the action of $\langle\si\rangle$. Let $A(t)=R(\si,t)$ denote the corresponding generalized Weyl algebra.

\begin{enumerate}[{\rm (a)}]
\item 
For each $t$-admissible directed subset $S$ of $\mathbb{R}$ there is a unique (up to isomorphism) semi-indecomposable object $M_S^t$ in $A(t)\mathrm{-wmod}_\CO$ such that
(i) $\Supp(M_S^t)=S\cap \Z$;
(ii) If $v\in M_S^t$ is a nonzero weight vector of weight $k_i-\frac{1}{2}$ then $X^+v\neq 0$ iff $S(k_i)=\rightarrow$;
(iii) If $v\in M_S^t$ is a nonzero weight vector of weight $k_i+\frac{1}{2}$ then $X^-v\neq 0$ iff $S(k_i)=\leftarrow$;

\item The assigment $S\mapsto [M_S^t]$ is a bijection between the set of
$t$-admissible directed subset $S$ of $\mathbb{R}$
and the set of isomorphism classes of semi-indecomposable objects
in $A(t)\mathrm{-wmod}_\CO$. Moreover, $M^t_S$ is indecomposable iff the underlying set of $S$ is connected.

\item If $t$ and $t'$ are two monic polynomials in $\C[z]$ with zero sets being finite subsets of $\Z+\frac{1}{2}$, and if $S\in\mathcal{P}^{\mathrm{dir}}_t(\mathbb{R})$ and $S'\in \mathcal{P}^{\mathrm{dir}}_{t'}(\mathbb{R})$ then
\begin{equation} \label{eq:tensor-product-rule}
M_S^t\otimes_{\C[z]} M_{S'}^{t'} \cong M_{S\cap S'}^{tt'}
\end{equation}
as $A(tt')$-modules.
\item For any submonoid $\Ga$ of the monoid of monic polynomials with half-integer roots, there is a $\C$-algebra isomorphism
\begin{align}
\mathscr{A}^{\mathrm{split}}(\Gamma,\mathcal{O})
&\cong \C[x_{S,t}\mid t\in \Ga, \; S\in\mathcal{P}^{\mathrm{dir}}_t(\R)]/(\mathrm{Rels})\\
K_0^{\mathrm{split}}(A(t)\mathrm{-wmod}_\CO)\ni[M_S^t] &\mapsto x_{S,t}
\end{align}
where the relations are given by
\begin{align}
x_{S,t}x_{T,t'} &=x_{S\cap T,tt'}  \label{eq:rels-type-1}\\
x_{S,t}+x_{T,t}&=x_{S\cup T,t} \quad\text{when $S$ and $T$ are strongly disjoint.} \label{eq:rels-type-2}
\end{align}
\end{enumerate}
\end{Theorem}

\begin{proof}
(a) and (b) are immediate by the classification of indecomposable weight modules over GWAs from \cite{DGO}.

(c) Put $W=M_S^t\otimes_{\C[z]} M_{S'}^{t'}$. By part (a) it suffices to check that $W$ satisfies properties (i)-(iii) from part (a) with respect to the $tt'$-admissible directed subset $S\cap S'$. By Lemma \ref{lemma:weightmod}, $W$ is a weight module with support equal to $\Supp\big(M_{S}^{t}\big)\cap \Supp\big(M_{S'}^{t'}\big)=(S\cap \Z)\cap (S'\cap\Z)=(S\cap S')\cap \Z$. Suppose $v\otimes v'$ is a nonzero weight vector in $W$ of weight $k_i-\frac{1}{2}$. Then $X^+(v\otimes v')=(X^+v)\otimes(X^+v')$ which is nonzero iff $S(k_i)=\rightarrow$ and $S'(k_i)=\rightarrow$. By our definition of intersection of directed subsets, this is equivalent to $(S\cap S')(k_i)=\rightarrow$. Similarly $X^-(v\otimes v')$ is nonzero iff $(S\cap S')(k_i)=\leftarrow$. 

(d) Let $B= \C[x_{S,t} \mid t\in \Ga, S\in\mathcal{P}^{\mathrm{dir}}_t(\R)]$ and define a map 
\[
\Phi: B\to\mathscr{A}^{\mathrm{split}}(\Ga,\mathcal{O})
\]
by
\[x_{S,t}\mapsto [M_S^t].\]
$\Phi$ is surjective: Every object is a finite sum of indecomposables. Every indecomposable is of the form $M_S^t$ by part (a).

(Rels)$\subseteq \ker\Phi$: Relations \eqref{eq:rels-type-1} belong to $\ker\Phi$ by \eqref{eq:tensor-product-rule}.
If $S$ and $T$ are strongly disjoint $t$-admissible directed subsets of $\R$, then $M^t_{S\cup T}\cong M^t_S\oplus M^t_T$ by the fact that $\Supp(M^t_{S\cup T})=Z\cap(S\cup T)=(\Z\cap S)\cup (\Z\cup T)=\Supp(M^t_S\oplus M^t_T)$ using part (a). So Relations \eqref{eq:rels-type-2} also belong to $\ker(\Phi)$.

Thus we get an induced surjective map
\[
\tilde\Phi: B/(\mathrm{Rels})\to \mathscr{A}^{\mathrm{split}}(\Ga,\mathcal{O}).
\]

$\tilde\Phi$ is injective: Since $\tilde\Phi$ is a map of $\Ga$-graded algebras, it suffices to show that $\tilde\Phi$ is injective on each homogeneous component $B_t$ for $t\in \Ga$. We define an inverse map $\Psi_t:\mathscr{A}^{\mathrm{split}}(\Ga,\mathcal{O})_t\to B_t$ as follows. 
We have $\mathscr{A}^{\mathrm{split}}(\Ga,\mathcal{O})_t=K_0^{\mathrm{split}}(A(t)\textrm{-wmod}_\CO)$ which is a free abelian group on the set of isoclasses of indecomposables in $A(t)\textrm{-wmod}_\CO$. By part (a), any such indecomposable module is of the form $M_S^t$ where $S$ is a $t$-admissible directed subset of $\R$ whose underlying set is connected. Define $\Psi_t([M_S^t])=x_{S,t}$ and extend additively. We have $\Psi_t\tilde\Phi(x_{S,t})=x_{S,t}$ which proves that $\tilde\Phi$ is injective.
\end{proof}

\begin{Lemma} \label{lem:a-lemma-about-directed-sets}
Let $t=\prod_{i=1}^r (z-k_i)^{m_i}$ be any monic polynomial with half-integer roots $k_i\in\Z+\frac{1}{2}$, $k_1<k_2<\cdots<k_r$. Then any connected $t$-admissible directed subset $S$ of $\R$ can be written as an intersection
\begin{equation}
S=S_1\cap S_2\cap \cdots \cap S_r
\end{equation}
where $S_i$ is $(z-k_i)$-admissible and
where the underlying set of $S_i$ is one of $(k_i,\infty)$, $(-\infty,k_i)$, $\R$.
Moreover, this decomposition is unique, if we choose $(k_i,\infty)$ or $(-\infty,k_i)$ over $\R$, when possible.
\end{Lemma}

\begin{proof}
We have $\overset{\circ}{S}=(k_i,k_j)$ where $0\le i<j\le r+1$ where we put $k_0=-\infty$ and $k_{r+1}=\infty$. 
For $1\le a\le i$, define $S_a$ to be the (undirected) subset $(k_a,\infty)$.
For $i<a<j$, define $S_a$ by $\overset{\circ}{S}_a=\R$ and $S_a(k_a)=S(k_a)$ and remaining elements of $\R$ undirected.
Finally, for $j\le a\le r$, define $S_a$ to be the (undirected) subset $(-\infty,k_a)$. Then clearly $S= S_1\cap S_2\cap\cdots\cap S_r$. For $i<a<j$, the choice of $S_a$ is unique. In the cases $1\le a\le i$, the only other choice of $S_a$ would be $\overset{\circ}{S}_a=\R$ with $S_a(k_i)\in\{1,\rightarrow,\leftarrow\}$. Similarly on the right side. So choosing the directed subsets whose underlying set are the half-open intervals instead of $\R$, we get uniqueness.
\end{proof}

Recall the simple modules $M_{k,\infty}^{z-k}$ and $M_{-\infty,k}^{z-k}$ from the previous section. We denote them here by $M_{(k,\infty)}^{z-k}$ and $M_{(-\infty,k)}^{z-k}$ to match the notation of Theorem \ref{thm:a-theorem-about-modules}. Also, for $k\in\Z+\frac{1}{2}$, let $\R_k^\pm$ denote the $(z-k)$-admissible directed subset whose underlying set is $\R$ and $k\overset{\rightarrow}{\in}\R_k^+$, while
$k\overset{\leftarrow}{\in}\R_k^-$.

\begin{Proposition}\label{prp:tensor-decomposition-of-indecomposables}
Let $t=\prod_{i=1}^r (z-k_i)^{m_i}$ be any monic polynomial with half-integer roots $k_i\in\Z+\frac{1}{2}$, $k_1<k_2<\cdots<k_r$.
Let $M$ be any indecomposable object in $A(t)\mathrm{-wmod}_\CO$, and let $S$ be the corresponding connected $t$-admissible directed subset of $\R$ such that $M\cong M_S^t$. Then
\begin{equation}\label{eq:tensor-decomposition}
M\cong M_1^{\otimes m_1}\otimes M_2^{\otimes m_2}\otimes\cdots\otimes M_r^{\otimes m_r}
\end{equation}
where $\otimes=\otimes_{\C[z]}$,
and $M_i$ are indecomposable modules over the Weyl algebra $A(z-k_i)$, given as follows. Let $0\le i<j\le r+1$ be such that $\overset{\circ}{S}=(k_i,k_j)$, where $k_0=-\infty$ and $k_{r+1}=\infty$. Then

\begin{equation}
M_a =
\begin{cases}
 M_{(-\infty,k_a)}^{z-k_a} & \text{$1\le a\le i$},\\
 M_{\R^+_{k_a}}^{z-k_a} & \text{if $i<a<j$ and $S(k_a)=\rightarrow$},\\
 M_{\R^-_{k_a}}^{z-k_a} & \text{if $i<a<j$ and $S(k_a)=\leftarrow$},\\
 M_{(k_a,\infty)}^{z-k_a} & \text{$j\le a\le r$}.
\end{cases}
\end{equation}
Moreover, the modules $M_a$ are uniquely determined for $i<a<j$. For $1\le a\le i$ and for $j\le a\le r$, the module $M_a$ could be replaced by $M_{\R^\pm_{k_a}}^{z-k_a}$ and the isomorphism \eqref{eq:tensor-decomposition} would still hold.
\end{Proposition}

\begin{proof}
The existence of the directed subset $S$ such that $M\cong M_S^t$ follows from Theorem \ref{thm:a-theorem-about-modules}(b). By Lemma \ref{lem:a-lemma-about-directed-sets}, $S$ can be decomposed as $S_1\cap S_2\cap \cdots\cap S_r$. By repeated use of Theorem \ref{thm:a-theorem-about-modules}(c), this yields the required decomposition.
\end{proof}

\begin{Theorem}
If $Z$ is a subset of $\Z+\frac{1}{2}$ and $\Ga$ is the monoid of all monic polynomials whose zero-set is contained in $Z$, then there is a $\C$-algebra isomorphism
\begin{equation}
\mathscr{A}^{\mathrm{split}}(\Gamma,\mathcal{O})
\cong \C[x^\pm_k,y^\pm_k\mid k\in Z]/(\mathrm{Rels})
\end{equation}
where the relations are given by:
\begin{subequations}\label{eq:relations-in-the-split-case}
\begin{align}
y_j^+x_k^+&=y_j^-x_k^+=x_j^+x_k^+\quad \text{if $j<k$,}\\
x_j^-y_k^+&=x_j^-y_k^-=x_j^-x_k^-\quad \text{if $j<k$,}\\
x_j^-x_k^+&=0\quad \text{if $j\le k$,}\\
y_k^+ y_k^- &= (x_k^+)^2 + (x_k^-)^2,
\end{align}
\end{subequations}
for all $j,k\in Z$.
The isomorphism maps 
the generators $x_k^\pm$ and $y_k^\pm$ 
to the two simples and the two non-simple indecomposables of 
$K_0^{\mathrm{split}}(A(z-k)\mathrm{-wmod}_\CO)$
respectively. Explicitly,
\begin{subequations}\label{eq:split-mapping-def}
\begin{align}
x_k^+ &\mapsto [M_{(k,\infty)}^{z-k}]\\
x_k^- &\mapsto [M_{(-\infty,k)}^{z-k}]\\
y_k^+ &\mapsto [M_{\R_k^+}^{z-k}]\\
y_k^- &\mapsto [M_{\R_k^-}^{z-k}]
\end{align}
\end{subequations}
where $\R_k^\pm$ is just the set $\R$ except the point $k$ is directed: $\R_k^+(k)=\rightarrow$ and $\R_k^-(k)=\leftarrow$.

\end{Theorem}

\begin{proof}
Let $B=\C[x_k^\pm,y_k^\pm\mid k\in Z]$ and define a $\C$-algebra map $\Phi:B\to \mathscr{A}^{\mathrm{split}}(\Ga,\CO)$ by \eqref{eq:split-mapping-def}. By Proposition \ref{prp:tensor-decomposition-of-indecomposables}, $\Phi$ is surjective and the relations \eqref{eq:relations-in-the-split-case} belong to the kernel of $\Phi$, inducing a surjection $\tilde\Phi:B/(\mathrm{Rels})\to \mathscr{A}^{\mathrm{split}}(\Ga,\CO)$. The injectivity of $\tilde\Phi$ follows from the normal form of words in $\bar B=B/(\mathrm{Rels})$. In more detail, let $t=(z-k_1)^{m_1}\cdots (z-k_r)^{m_r}$ with $k_i\in\Z+\frac{1}{2}$, $k_1<\ldots<k_r$ and $m_i\in\Z_{>0}$.
The set of elements of $\bar B_t$ of the form
\[ X^+_i Y^\ep X^-_j \]
where $0\le i<j\le r+1$ and $\ep=(\ep_{i+1},\ldots,\ep_{j-1}) \in\{+,-\}^{j-i-1}$, and
\[X^+ = (x_{k_1}^+)^{m_1}\cdots (x_{k_i}^+)^{m_i} \]
\[Y = (y_{k_{i+1}}^{\ep_{i+1}})^{m_{i+1}}\cdots (y_{k_{j-1}}^{\ep_{j-1}})^{m_{j-1}}\]
\[X^- = (x_{k_j}^-)^{m_j}\cdots (x_{k_r}^-)^{m_r}\]
is a basis for $\bar B_t$ as a vector space over $\C$. Indeed, that these span can be checked by induction on $m_1+\cdots+m_r$, and their linear independence follow from the Diamond Lemma.
These correspond precisely to the indecomposable objects in $A(t)\mathrm{-wmod}_\CO$ which in turn form a $\C$-basis for $\C\otimes_\Z K_0^\mathrm{split}(A(t)\mathrm{-wmod}_\CO)=\mathscr{A}^\mathrm{split}(\Ga,\CO)_t$.
\end{proof}

We can also describe the algebra map from the split to the non-split algebra.
\begin{Theorem} If $Z$ is a subset of $\Z+\frac{1}{2}$ and $\Ga$ is the monoid of all monic polynomials whose zero-set is contained in $Z$, then the
canonical homomorphism \eqref{eq:canon} from the split Grothendieck group to the Grothendieck group yields a surjective $\C$-algebra homomorphism
\begin{equation}
\mathscr{A}^{\mathrm{split}}(\Gamma,\mathcal{O})
\to \mathscr{A}(\Gamma,\mathcal{O})
\end{equation}
which in terms of the algebra generators is given by
\begin{align}
x_k^+ &\mapsto x_k^+\\
x_k^- &\mapsto x_k^-\\
y_k^\pm &\mapsto x^+_k+x^-_k
\end{align}
for all $k\in Z$.
\end{Theorem}

\begin{proof}
$x_k^\pm$ correspond to the simple modules $M_{(-\infty,k)}^{(z-k)}$, $M_{(k,\infty)}^{(z-k)}$ respectively, while $y_k^\pm$ correspond to the indecomposable modules $M_{\R^\pm_k}^{(z-k)}$. The latter have composition series of length two:
\begin{equation}
0\to M_{(k,\infty)}^{(z-k)} \to M_{\R^+_k}^{(z-k)} \to
M_{(-\infty,k)}^{(z-k)}\to 0
\end{equation}
\begin{equation}
0\to M_{(-\infty,k)}^{(z-k)} \to M_{\R^+_k}^{(z-k)} \to
M_{(k,\infty)}^{(z-k)}\to 0
\end{equation}
This proves the claim.
\end{proof}


\section{Rank two setting (Cylinder)}\label{sec:cylinder}

Now we describe the situation of Section \ref{section:towers} in the special case of $n=2$, $R=\C[z]$, $\sigma_i^{1/2}(z)=z-\frac{\alpha_i}{2}$, $\alpha_i\in\C$, $i=1,2$. In this case, the solutions to the consistency equation \eqref{eq:cons1} were classified in \cite{HarRos16}. In particular, we have nontrivial solutions $t=(t_1,t_2)\in\Omega$ if and only if $\frac{\alpha_1}{\alpha_2}$ is a negative rational number. For simplicity, we will then, throughout this section, fix a pair or relatively prime positive integers $m,n$ and assume that $\alpha_1=-n$, $\alpha_2=m$. 

\begin{Remark}We have $\MaxSpec(R)=\{(z-\la)~|~\la\in\C\}$ with $\Z^2$-action given by $\sigma_1(z-\la)=(z-\la+n)$, $\sigma_2(z-\la)=(z-\la-m)$. Clearly $\Z^2\cdot (z-\la)\subset \{(z-\la+\Z)\}$, but since $m,n$ are relatively prime, we actually have $\Z^2\cdot (z-\la)= \{(z-\la+\Z)\}$. Hence the orbits of this action can again be parametrized by $\C/\Z$. If $\la+\Z\in\C/\Z$, the corresponding orbit is $\mathcal{O}_{\la+\Z}=\{(z-\la+\Z)\}$. Notice that, unlike the case of Section \ref{sec:line}, this $\Z^2$ orbit is \emph{not} torsion free, hence some of the results from Section \ref{section:towers} do not apply here. It is however still true that all the orbits are isomorphic, hence we will specifically focus only on $\CO_\Z=\{(z-\la)~|~\la\in\Z\}$.
\end{Remark}

We now describe $\Omega$ using the conventions of \cite{H2018}.


\begin{Definition}\label{def:cylinder}We consider the quotient group $C=\mathbb{R}^2/G$, where $G$ is the additive subgroup of $\R^2$ generated by $(m,n)\in\mathbb{R}^2$. We call $C$ the cylinder because, as a topological space, we have a homeomorphism $C\simeq S^1\times\mathbb{R}$. We define certain discrete subsets of $C$. 
\begin{itemize}
\item The \emph{face lattice} of the cylinder is $L=\Z^2/G\subset C$. 
\item The two \emph{edge lattices} are
$$ E_i=\left(\tfrac{1}{2}\bold{e}_i+\Z^2\right)/G\subset C, \qquad i=1,2$$
where $\bold{e}_1=(1,0)$, $\bold{e}_2=(0,1)$.
\item The \emph{vertex lattice} is $V=\left(\tfrac{1}{2}+\Z\right)^2/G\subset C$.
\end{itemize}
\end{Definition}

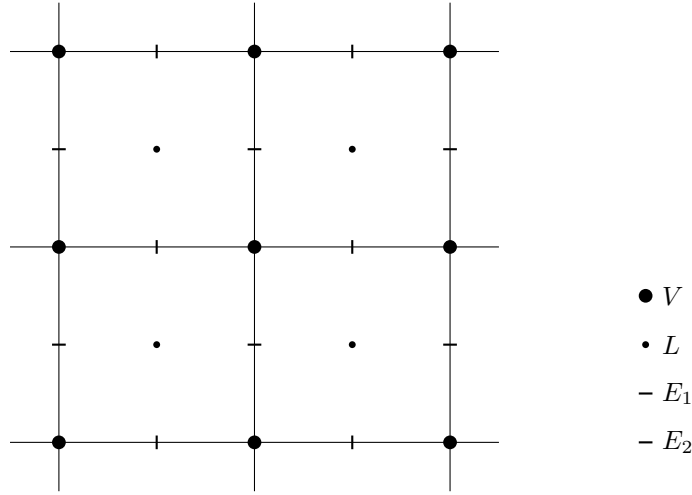
\begin{figure}
\centering
\begin{tikzpicture}[scale=1.3]
\foreach \y in {0,2,4} {
\draw (-.5 cm,\y cm) -- (4.5 cm,\y cm);}
\foreach \x in {0,2,4} {
\draw (\x cm,-.5 cm) -- (\x cm,4.5 cm);}
\foreach \y in {1,3} {
 \foreach \x in {1,3} {
  \fill (\x,\y) circle (1pt);}}
\foreach \y in {0,2,4} {
 \foreach \x in {0,2,4} {
  \fill (\x,\y) circle (2pt);}}
\foreach \x in {0,2,4} {
 \foreach \y in {1,3} {
  \draw[thick] (\x-.07,\y) -- (\x+.07,\y); }}
\foreach \x in {1,3} {
 \foreach \y in {0,2,4} {
\draw[thick] (\x,\y-.07) -- (\x,\y+.07);}}
\fill (6,1.5) circle (2pt);
\draw (6,1.5) node[right] {$\;V$};
\fill (6,1) circle (1pt);
\draw (6,1) node[right] {$\;L$};
\draw[thick] (6-.07,.5) -- (6+.07,.5);
\draw (6,.5) node[right] {$\;E_1$};
\draw[thick] (6-.07,0) -- (6+.07,0);
\draw (6,0) node[right] {$\;E_2$};
\end{tikzpicture}
\caption{Square lattice grids.}
\label{fig:lattice}
\end{figure}

\begin{Definition}\label{def:config}
A \emph{configuration} on $C$ is a function $\omega=(\omega_1,\omega_2):E_1\times E_2\to\N$ such that two conditions are satisfied:
\begin{enumerate}
\item $|\omega^{-1}(\N\setminus\{0\})|<\infty$ (finiteness)
\item For all $v\in V$ we have the \emph{ice rule}
\begin{equation}\label{eq:icerule} \omega\left(v-\tfrac{1}{2}\bold{e}_1\right)+\omega\left(v-\tfrac{1}{2}\bold{e}_2\right)=\omega\left(v+\tfrac{1}{2}\bold{e}_1\right)+\omega\left(v+\tfrac{1}{2}\bold{e}_2\right).\end{equation}
\end{enumerate}
We denote the set of all configurations on $C$ by $\mathscr{C}$.
\end{Definition}
\begin{Remark}In Definition \ref{def:cylinder}, we have identified the edges of the edge lattices by their midpoints, it will be useful however to also consider them as line segments of length one. Specifically, if $s_1=\left(a+\frac{1}{2},b\right)\in E_1$, the corresponding line segment is $\{(x,y)\in\mathbb{R}^2/G~|~x=a+\frac{1}{2},~b-\frac{1}{2}\leq y\leq b+\frac{1}{2}\}$. Analogously, if $s_2=(a,b+\frac{1}{2})\in E_2$, we have the corresponding segment $\{(x,y)\in\mathbb{R}^2/G~|~a-\frac{1}{2}\leq x\leq a+\frac{1}{2},~= b+\frac{1}{2}\}$. In particular, if $\omega\in \mathscr{C}$ is a configuration, we consider $\Supp(\omega)=\omega^{-1}(\N\setminus\{0\})$ as a union of line segments. Notice that we are using $\Supp$ both to denote the support of a weight module and the support of a configuration, the context should prevent any possible confusion.
\end{Remark}
\begin{Remark}
A function $\omega:E_1\times E_2\to\N$ is a configuration on $C$ if and only if it can be written as a sum of indicator functions of $(m,n)$-paths in $C$ (lattice paths of length $m+n$ with $m$ steps east and $n$ steps north, due to the identification of the cylinder these are actually loops topologically). We will identify the paths with their indicator functions throughout this section.
\end{Remark}
\begin{Example}Let $m=3$, $n=2$. We draw the cylinder as a fundamental domain of an infinite vertical strip in the plane, we identify points $(3,y)$ on the right boundary with points $(0,y-2)$ on the left boundary. An example of a configuration on the cylinder is in figure \ref{fig:configuration} on the left, where we write the value assigned to each edge (zero if nothing is written). On the right we show one possible way of decomposing the configuration as a sum of four $(3,2)$-paths. This is in fact the unique way to write the configuration in such a way that the resulting paths are a chain in the partial order defined on the set of $(m,n)$-paths on the cylinder. The partial order is defined, for two paths $\pi$ and $\nu$, by $\pi\geq \nu$ if each segment of $\pi$ either overlaps with $\nu$ or it is to the north of $\nu$.
\end{Example}
\begin{figure}
\centering
\begin{tikzpicture}[scale=1]
\foreach \y in {0,1,2,3,4,5} {
\draw (0,\y) -- (3,\y);}
\foreach \x in {0,1,2} {
\draw (\x,-1) -- (\x,6);}
\draw[dashed] (3,-1) -- (3,6);
\draw[red] (0,1.5) node {$2$};
\draw[red] (0,2.5) node {$1$};
\draw[red] (0,3.5) node {$1$};
\draw[red] (0.5,2) node {$3$};
\draw[red] (0.5,4) node {$1$};
\draw[red] (1,2.5) node {$2$};
\draw[red] (1,3.5) node {$1$};
\draw[red] (1.5,2) node {$1$};
\draw[red] (1.5,3) node {$1$};
\draw[red] (1.5,4) node {$2$};
\draw[red] (2,2.5) node {$1$};
\draw[red] (2.5,3) node {$2$};
\draw[red] (2.5,4) node {$2$};
\end{tikzpicture}\hspace{4cm}
\begin{tikzpicture}[scale=1]
\foreach \y in {0,1,2,3,4,5} {
\draw (0,\y) -- (3,\y);}
\foreach \x in {0,1,2} {
\draw (\x,-1) -- (\x,6);}
\draw[dashed] (3,-1) -- (3,6);
\draw[red,thick] (0,2) -- (0,4.02);
\draw[red,thick] (0,4.02) -- (3,4.02);
\draw[blue,thick] (0,2.03) -- (0.98,2.03) -- (0.98,3.98) -- (3,3.98);
\draw[green,thick] (-0.02,1) -- (-0.02,2) -- (1.02,2)--(1.02,3.02) -- (3,3.02);
\draw[orange,thick] (0.02,1) -- (0.02,1.97) -- (2,1.97) -- (2,2.98) -- (3,2.98);
\end{tikzpicture}
\caption{Configuration and corresponding paths.}
\label{fig:configuration}
\end{figure}
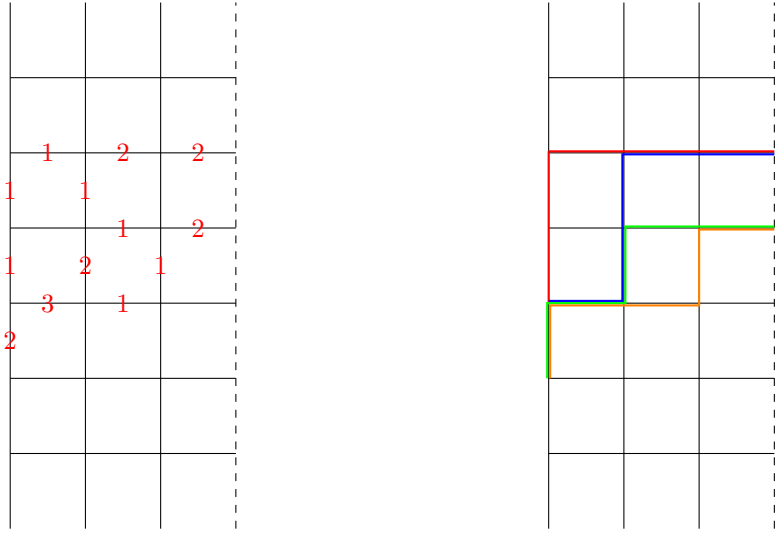
\begin{Remark}As described in \cite[Thm 1.15]{H2018}, for each $t\in\Omega$, we have a set of pairs $\{(\omega^{(j)},\la_j)\in\mathscr{C}\times\C~|~j=1,\ldots,k\}$ such that
$$t=\prod_{j=1}^k t^{\omega^{(j)}} $$
where 
$$t^{\omega^{(j)}}=(t^{\omega^{(j)}}_1,t^{\omega^{(j)}}_2)\in\Omega,\qquad t^{\omega^{(j)}}_i=\prod_{(s_1,s_2)\in E_i}(z+s_1n-s_2m-\la_j)^{\omega^{(j)}_i(s_1,s_2)}, \quad i=1,2.$$
Notice that for all $j=1,\ldots,k$, $t^{\omega^{(j)}}_i\in\C[z]$ because of the finiteness condition \ref{def:config}(1) and $t^{\omega^{(j)}}$ satisfies \eqref{eq:cons1} because of the ice rule \eqref{eq:icerule}.
Additionally, the pairs $(\omega^{(j)},\la_j)$ corresponding to a given $t$ are uniquely determined up to the $\Z^2$-action. Notice that each $t^{\omega^{(j)}}$ is an \emph{orbital solution}, in the terminology of \cite{HarRos20}.
\end{Remark}
For any $t\in\Omega$, we have a rank two TGWA over $\C[z]$, denoted $A(t)$, with generators $X_1^{\pm}$, $X_2^{\pm}$ and relations given as in Definition \ref{def:tgwa}.
Even though we do not have torsion-free orbits in this setting, it is still true that, by the localization results of Theorem \ref{thm:wmodO}(b), when examining the category $A(t)\text{-wmod}_\CO$, we can restrict ourselves to parameters $t=(t_1,t_2)$ such that all the irreducible factors of $\si_i^{1/2}(t_i)$ are in $\bigcup_{\Fm\in\CO}\Fm$. This is basically the idea of \cite[Theorem A]{H2018}. We fix $\CO=\CO_\Z$ and accordingly we then define
$$\Omega'=\left\{t^\omega=(t^\omega_1,t^\omega_2)\in\Omega~|~t^\omega_i=\prod_{(s_1,s_2)\in E_i}(z+s_1n-s_2m)^{\omega_i(s_1,s_2)},~i=1,2,~\text{ for some }\omega\in\mathscr{C}\right \}.$$

\subsection{The algebra $\scA(\Omega',\CO)$}
We recall here the classification of simple weight modules for $A(t^\omega)$, which was one of the main results of \cite{H2018}, adapting the notation to our current setting.
\begin{Theorem}[{\cite[Thm. B]{H2018}}]\label{thm:simples-cyl}Let $t^\omega\in\Omega'$, then the simple objects of $A(t^\omega)\text{-wmod}_\CO$ are $M^{\omega}_{(D,\xi)}$ where $D$ is a connected component of $C\setminus\Supp(\omega)$ and $\xi\in\C$ is such that $\xi=0$ if and only if $D$ is contractible. Moreover, $\Supp(M^\omega_{(D,\xi)})=L\cap D$, all the weight spaces are one dimensional over $\C$, and there is a centralizing element $c(\omega)\in A(t^\omega)\otimes_{\C[z]} \C(z)$ such that $c(\omega)$ acts by $\xi$ on $M^\omega_{(D,\xi)}$ when $D$ is not contractible.
\end{Theorem}
\begin{Example}
For the configuration $\omega$ of Figure \ref{fig:configuration}, there are five connected components of $C\setminus\Supp(\omega)$, two unbounded ones, which are not contractible, and three bounded ones that are contractible.
\end{Example}
\begin{Proposition}\label{prop:central}
For any sequence $\underline{i}=(i_1,i_2,\ldots,i_{m+n})$ of $m$ $1$'s and $n$ $2$'s, there exists a nonzero polynomial $f^{\underline{i}}_{\omega}(z)\in\C[z]$
such that the centralizing element $c(\omega)\in A(t^\omega)\otimes_{\C[z]} \C(z)$ equals
$$ c(\omega)=X(\underline{i}) \cdot\frac{1}{f^{\underline{i}}_{\omega}(z)}$$
where $X(\underline{i})=X_{i_{m+n}}^+\cdots X_{i_2}^+X_{i_1}^+$. Moreover, if $D$ is not contractible then, for all $(z-\lambda)\in\Supp(M^\omega_{(D,\xi)})$ and corresponding sequence $\underline{i}$ such that the face path $\lambda, \lambda+\boldsymbol{e}_{i_1}, \ldots, \lambda+\boldsymbol{e}_{i_1}+\cdots+\boldsymbol{e}_{i_{m+n}} $ does not cross any edge from $\omega$, we have $f^{\underline{i}}_{\omega}(\lambda)\neq 0$ and for all $m\in (M^\omega_{(D,\xi)})_{(z-\lambda)}$ the action is defined by 
$$ c(\omega)\cdot m=X(\underline{i})\cdot \left(\frac{1}{f^{\underline{i}}_{\omega}(\lambda)}m\right).$$
Finally, for any sequence $\underline{i}$, the polynomials $f^{\underline{i}}_{\omega}(z)$ satisfy
$$ f^{\underline{i}}_{\omega+\omega'}(z)=f^{\underline{i}}_{\omega}(z)\cdot f^{\underline{i}}_{\omega'}(z).$$
\end{Proposition}
\begin{proof}
This follows directly from \cite[Prop. 6.3]{H2018}, by using the properties of $\operatorname{ord}(\underline{i},\lambda)$.
\end{proof}
\begin{Remark}\label{rem:inters-comp}
If $D$ is a connected component of $C\setminus\Supp(\omega)$ and $D'$ is a connected component of $C\setminus\Supp(\omega')$, then $D\cap D'\subseteq C\setminus\Supp(\omega+\omega')$ is either a single noncontractible connected component (which is only possible if $D$ and $D'$ were both noncontractible), or a disjoint union of contractible components. We denote the set of connected components of $C\setminus\Supp(\omega)$ by $H_0(C\setminus\Supp(\omega))$ (zeroth homology group).
\end{Remark}
\begin{Proposition}\label{prop:prod-nonct}
Consider $M^\omega_{(D,\xi)}$, $M^{\omega'}_{(D',\xi')}$, with $D,D'$ not contractible such that $D\cap D'$ is also not contractible.
Then we have an isomorphism of $A(t^{\omega+\omega'})$-modules
$$M^\omega_{(D,\xi)}\otimes_{\C[z]} M^{\omega'}_{(D',\xi')}\simeq M^{\omega+\omega'}_{(D\cap D',\xi\cdot \xi')}.$$
\end{Proposition}
\begin{proof}
First of all, it is clear by Lemma \ref{lemma:weightmod} that
$$ \Supp (M^\omega_{(D,\xi)}\otimes_{\C[z]} M^{\omega'}_{(D',\xi')})=\Supp (M^\omega_{(D,\xi)})\cap \Supp(M^{\omega'}_{(D',\xi')})=D\cap D'$$
and all the weight spaces are one dimensional. Then, the only thing left to check is the action of $c(\omega+\omega')$ on the module $M^\omega_{(D,\xi)}\otimes_{\C[z]} M^{\omega'}_{(D',\xi')}$. Let $m\otimes m'\in (M^\omega_{(D,\xi)}\otimes_{\C[z]} M^{\omega'}_{(D',\xi')})_{(z-\lambda)}= (M^\omega_{(D,\xi)})_{(z-\lambda)}\otimes_{\C[z]} (M^{\omega'}_{(D',\xi')})_{(z-\lambda)}$. Choose a sequence $\underline{i}$ such that path in the face lattice $\lambda, \lambda+\boldsymbol{e}_{i_1},\ldots, \lambda+\boldsymbol{e}_{i_1}+\cdots+\boldsymbol{e}_{i_{m+n}}$ does not cross any edges from the configuration $\omega+\omega'$. Then the same path does not cross any edges from $\omega$ nor from $\omega'$. We have
\begin{align*}
c(\omega+\omega')\cdot m\otimes m'&= X(\underline{i})\frac{1}{f^{\underline{i}}_{\omega+\omega'}(z)}\cdot m\otimes m' \\
&=  X(\underline{i})\cdot \left(\frac{1}{f^{\underline{i}}_{\omega+\omega'}(\lambda)} m\otimes m'\right) \\
\text{ (by Prop \ref{prop:central}) }&= \Delta_{t^\omega,t^{\omega'}}(X(\underline{i}))\left(\frac{1}{f^{\underline{i}}_{\omega}(\lambda)f^{\underline{i}}_{\omega'}(\lambda)}m\otimes m'\right)\\
&= \left(X(\underline{i})\otimes X(\underline{i})\right)\left(\frac{1}{f^{\underline{i}}_{\omega}(\lambda)}m\otimes \frac{1}{f^{\underline{i}}_{\omega'}(\lambda)} m'\right) \\
&= \left(X(\underline{i})\frac{1}{f^{\underline{i}}_{\omega}(\lambda)}m\right)\otimes \left(X(\underline{i})\frac{1}{f^{\underline{i}}_{\omega'}(\lambda)} m'\right)\\
&= (c(\omega)\cdot m )\otimes (c(\omega')\cdot m') \\
&= \xi m \otimes \xi' m' \\
&= \xi\xi'(m\otimes m').
\end{align*}
\end{proof}
Let $\Bzr\in \mathscr{C}$ defined by ${\bold 0}(e_1,e_2)=0$ for all $(e_1,e_2)\in E_1\times E_2$. Then $\Bbo=(1,1)=t^{\bold 0}\in\Omega'$ and $A(\Bbo)\simeq \C[z]\rtimes\Z[X_1^{\pm 1},X_2^{\pm 1}]$ as in Remark \ref{rem:unit} (although in this case the orbit is not torsion free so we do not have a unique simple module). Notice that $C\setminus\Supp({\bold 0})=C$ is connected and not contractible, and that $\Bbo$ is the trivial submonoid of $\Omega'$.
\begin{Proposition}\label{prop:cstar-grp-alg}
We have an algebra isomorphism 
$$\scA(\Bbo,\CO)\simeq \C[\C^\times]\qquad\text{ given by }\qquad [M^{\bold 0}_{(C,\xi)}]\mapsto \xi$$
where $\C[\C^\times]$ is the group algebra.
\end{Proposition}
\begin{proof}
By Theorem \ref{thm:simples-cyl}, we have 
$$K_0(A(\Bbo)\text{-wmod}_\CO)=\bigoplus_{\xi\in\C^\times}\Z [M^{\Bzr}_{(C,\xi)}].$$
In this case, $C\cap C=C$ is not contractible, hence by Proposition \ref{prop:prod-nonct}, we have 
$$ M^{\Bzr}_{(C,\xi)}\otimes M^{\Bzr}_{(C,\zeta)}\simeq M^{\Bzr}_{(C,\xi\zeta)}$$
and the result follows.
\end{proof}

\begin{Proposition}\label{prop:cyl-unital}
The algebra $\scA(\Omega',\CO)$ is unital, with unit element $[M^{\Bzr}_{(C,1)}]$.
\end{Proposition}
\begin{proof}Let $M^\omega_{(D,\xi)}\in A(t^\omega)\text{-wmod}$, then $\Supp(M^\omega_{(D,\xi)}\otimes_{\C[z]}M^{\Bzr}_{(C,1)})=D\cap C=D$. If $D$ is contractible, this is enough to prove that 
\begin{equation}\label{eq:produnit}M^\omega_{(D,\xi)}\otimes_{\C[z]}M^{\Bzr}_{(C,1)} \simeq M^\omega_{(D,\xi)}\end{equation}
 as a module for $A(t^\omega\cdot \Bbo)=A(t^\omega)$, with $\xi=0$, since there is a unique simple module with the given support. If $D$ is not contractible, the isomorphism \eqref{eq:produnit} of $A(t^\omega)$-modules also holds (with $\xi\neq 0$), due to Proposition \ref{prop:prod-nonct}.
\end{proof}
\begin{Theorem}
As $\Omega'$-graded vector spaces, we have an isomorphism
$$\scA(\Omega',\CO)\simeq \bigoplus_{\omega : t^\omega\in \Omega'}\bigoplus_{\substack{D\in H_0(C\setminus\Supp(\omega)),\\ \xi\in Z(D)}}\C[M^\omega_{(D,\xi)}]$$
with $$Z(D)=\begin{cases} \{0\} & \text{ if $D$ is contractible, } \\ \C^\times & \text{ if $D$ is noncontractible. }\end{cases}$$
Multiplication in the algebra is given by
$$[M^\omega_{(D,\xi)}]\cdot [M^{\omega'}_{(D',\xi')}]=\begin{cases} 0 & \text{ if }D\cap D=\emptyset, \\ [M^{\omega+\omega'}_{(D\cap D',\xi\xi')}] & \text{ if $D\cap D'$ is noncontractible, }\\
\sum_{D''\in H_0(D\cap D')}[M^{\omega+\omega'}_{(D'',0)}] & \text{ if $D\cap D'$ is a union of contractible components. } \end{cases}$$
\end{Theorem}
\begin{proof}
The statement about the basis for $\scA(\Omega',\CO)$ follows directly from Theorem \ref{thm:simples-cyl}. The multiplication formulas follow from the fact that $\Supp(M\otimes_{\C[z]}N)=\Supp(M)\cap \Supp(N)$, together with Proposition \ref{prop:prod-nonct} and Remark \ref{rem:inters-comp}.
\end{proof}
\begin{Theorem}\label{thm:pres-cyl}
$$\mathscr{A}(\Omega', \CO)\simeq \C[\gamma_\xi, x_{\pi}^\pm~|~\xi\in\C^\times,~\pi \text{ is a $(m,n)$-path }]/(\text{Rels})$$
where the relations are
\begin{enumerate}[{\rm (i)}]
\item$\gamma_{1}=1$,
\item$\gamma_{\xi_1}\gamma_{\xi_2}=\gamma_{\xi_1\xi_2}$,
\item\label{rel2}$ x_{\pi_1}^{\pm} x_{\pi_2}^{\pm}=x_{\pi'_1}^{\pm} x_{\pi'_2}^{\pm}$, \hspace{3.6cm} if $\pi_1+\pi_2=\pi'_1+\pi'_2$,
\item\label{rel1}$x_{\pi_1}^+ x_{\pi_2}^-=0$, \hspace{4.5cm}  if  $\pi_1\geq \pi_2$,
\item\label{rel2}$ x_{\pi_1}^+ x_{\pi_2}^-+x_{\pi_3}^+ x_{\pi_4}^-=x_{\pi'_1}^+ x_{\pi'_2}^-+x_{\pi'_3}^+ x_{\pi'_4}^-$, \hspace{0.7cm} if $\pi_1+\pi_2=\pi_3+\pi_4=\pi'_1+\pi'_2=\pi'_3+\pi'_4$ 

\hspace{5.6cm} and $\pi_1+\pi_3=\pi'_1+\pi'_3$,
\item$\gamma_{\xi}x_{\pi_1}^+ x_{\pi_2}^-=x_{\pi_1}^+ x_{\pi_2}^-$, \hspace{3.3cm} if $\pi_1$ and $\pi_2$ intersect.
\end{enumerate}
\end{Theorem}
\begin{Remark}\label{rem:joint-meet}Given two $(m,n)$-paths $\pi_1$ and $\pi_2$, there are two uniquely defined paths $\pi_1\vee\pi_2$ and $\pi_1\wedge\pi_2$ such that $(\pi_1\vee\pi_2)+(\pi_1\wedge\pi_2)=\pi_1+\pi_2$ and $(\pi_1\vee\pi_2)\geq\pi_1,\pi_2\geq(\pi_1\wedge\pi_2)$.
The relations (iv)-(v) then imply the following (by taking $\pi_3=\pi_2$, $\pi_4=\pi_1$, $\pi'_1=\pi'_4=\pi_1\wedge\pi_2$, $\pi'_2=\pi'_3=\pi_1\vee\pi_2$):
\begin{equation}\label{eq:wedge-vee} 
\mathrm{(vii)}\quad
x_{\pi_1}^+ x_{\pi_2}^-+x_{\pi_2}^+ x_{\pi_1}^-= x_{(\pi_1\wedge\pi_2)}^+ x_{(\pi_1\vee\pi_2)}^-.\end{equation}
\end{Remark}
\begin{Example}Let $m=3$, $n=4$, and consider the two paths $\pi_1$, $\pi_2$ on the left of Figure \ref{fig:vee-wedge}, then the paths $\pi_1\wedge\pi_2$ and $\pi_1\vee\pi_2$ are the ones on the right.
\end{Example}
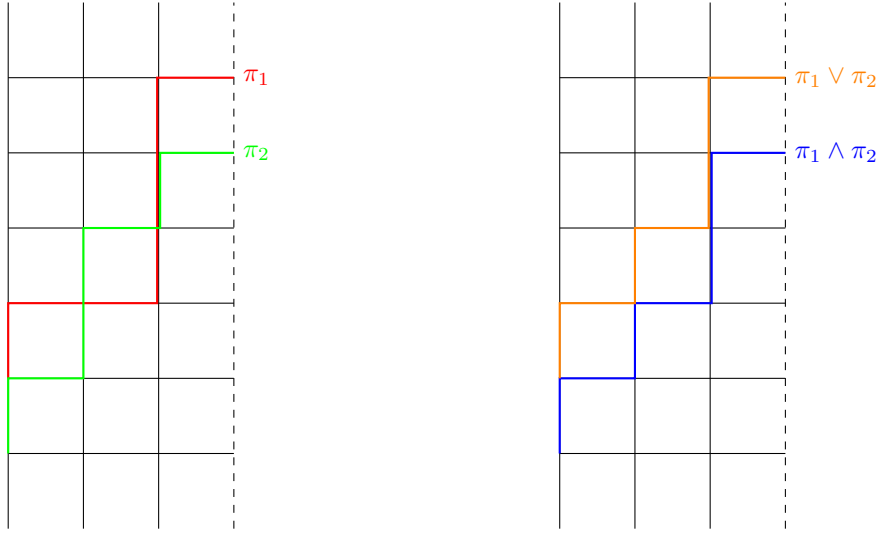
\begin{figure}
\centering
\begin{tikzpicture}[scale=1]
\foreach \y in {0,1,2,3,4,5} {
\draw (0,\y) -- (3,\y);}
\foreach \x in {0,1,2} {
\draw (\x,-1) -- (\x,6);}
\draw[dashed] (3,-1) -- (3,6);
\draw[red,thick] (0,1) -- (0,2) -- (1.98,2) -- (1.98,5) -- (3,5);
\draw[green,thick] (0,0) -- (0,1) -- (1,1) -- (1,3) -- (2.02,3) -- (2.02,4) -- (3,4);
\draw[red] (3,5) node[right] {$\pi_1$};
\draw[green] (3,4) node[right] {$\pi_2$}; 
\end{tikzpicture}\hspace{3.6cm}
\begin{tikzpicture}[scale=1]
\foreach \y in {0,1,2,3,4,5} {
\draw (0,\y) -- (3,\y);}
\foreach \x in {0,1,2} {
\draw (\x,-1) -- (\x,6);}
\draw[dashed] (3,-1) -- (3,6);
\draw[blue,thick] (0,0) -- (0,1) -- (1,1) -- (1,2) -- (2.02,2) -- (2.02,4) -- (3,4);
\draw[orange,thick] (0,1) -- (0,2) -- (1,2) -- (1,3) -- (1.98,3) -- (1.98,5) -- (3,5);
\draw[orange] (3,5) node[right] {$\pi_1\vee\pi_2$};
\draw[blue] (3,4) node[right] {$\pi_1\wedge\pi_2$}; 
\end{tikzpicture}
\caption{The paths $\pi_1$, $\pi_2$, $\pi_1\vee\pi_2$, and $\pi_1\wedge\pi_2$}
\label{fig:vee-wedge}
\end{figure}
\begin{proof}[Proof of Theorem \ref{thm:pres-cyl}]
Consider the map 
$$\alpha: \C[\gamma_\xi, x_{\pi}^\pm~|~\xi\in\C^\times,~\pi \text{ is a $(m,n)$-path }]\to \mathscr{A}(\Omega', \CO)$$
defined, for all $\xi\in\C^\times$ and for all $(m,n)$-paths $\pi$, by
$$ \gamma_\xi\mapsto [M^{\Bzr}_{C,\xi}],\quad x_{\pi}^\pm\mapsto [M^{\pi}_{D_{\pi}^\pm,1}]$$
where $D_{\pi}^+$ (resp. $D_{\pi}^-$) is the connected component of the cylinder above (resp. below) the path $\pi$ (since the path $\pi$ cuts the cylinder into two noncontractible components).

First we prove that this map is surjective. Let $\omega\in\mathscr{C}$, $D\in H_0(C\setminus\Supp(\omega))$. We write $\omega=\sum_{i=1}^k\pi_i$, where $\pi_i$ is an $(m,n)$-path on the cylinder and $\pi_1\geq \pi_2\geq \cdots\geq \pi_k$. We proceed by induction on $k$. If $k=0$, then $\omega=\Bzr$ and $[M^{\Bzr}_{D,\xi}]=\alpha(\gamma_\xi)$. If $k=1$, then 
$$[M^\omega_{D,\xi}]=[M^{\pi_1}_{D,\xi}]=[M^{\Bzr}_{C,\xi}\otimes_{\C[z]}M^{\pi_1}_{D,1}]=[M^{\Bzr}_{C,\xi}]\cdot[M^{\pi_1}_{D,1}]\in\Im(\alpha).$$ Now suppose that $k\geq 2$, and that the upper boundary (to the left and above) of $D$ is part of $\pi_s$, while the lower boundary (below and to the right) of $D$ is part of $\pi_{s+1}$. By swapping some edges of $\pi_s$ and $\pi_{s+1}$, if necessary, we can obtain paths $\pi'_s$ and $\pi'_{s+1}$ such that $\pi_s+\pi_{s+1}=\pi'_s+\pi'_{s+1}$, and $D=D_{\pi'_s}^+\cap D_{\pi'_{s+1}}^-$ (notice that $\pi_s'\not\geq\pi'_{s+1}$ in general). Let $\omega'=\omega-(\pi_s+\pi_{s+1})$, then there is $D'\in H_0(C\setminus\Supp(\omega'))$ such that $D\subseteq D'$. Then, if $D$ is contractible we have
$$ M^\omega_{D,0}\simeq M^{\omega'}_{D',\xi}\otimes_{\C[z]}M^{\pi'_s}_{D_{\pi'_s}^+,1}\otimes_{\C[z]}M^{\pi'_{s+1}}_{D_{\pi'_{s+1}}^-,1}$$
for any $\xi\in Z(D')$. If $D$ is non contractible we have, for all $\xi\in\C^\times$,
$$ M^\omega_{D,\xi}\simeq M^{\omega'}_{D',\xi}\otimes_{\C[z]}M^{\pi'_s}_{D_{\pi'_s}^+,1}\otimes_{\C[z]}M^{\pi'_{s+1}}_{D_{\pi'_{s+1}}^-,1}.$$
In either case, since $[M^{\omega'}_{D',\xi}]\in\Im(\alpha)$ by inductive hypothesis, it folllows that $[M^\omega_{D,0}]\in\Im(\alpha)$ (resp. for all $\xi\in\C^\times$, $[M^\omega_{D,\xi}]\in\Im(\alpha)$).

Then we want to prove that the relations are satisfied in $\scA(\Omega',\CO)$ so that the map $\alpha$  descends to a map on the quotient.
Relations (i) and (ii) follow directly from Prop. \ref{prop:cyl-unital} and Prop. \ref{prop:cstar-grp-alg}. 

If we have paths such that $\pi_1+\pi_2=\pi'_1+\pi'_2$, then $D_{\pi_1}^\pm\cap D_{\pi_2}^\pm=D_{\pi'_1}^\pm\cap D_{\pi'_2}^\pm$ are noncontractible, hence
$$M^{\pi_1}_{D_{\pi_1}^{\pm},1}\otimes_{\C[z]}M^{\pi_2}_{D_{\pi_2}^{\pm},1}\simeq M^{\pi_1+\pi_2}_{D_{\pi_1}^{\pm}\cap D_{\pi_2}^\pm ,1}= M^{\pi'_1+\pi'_2}_{D_{\pi'_1}^{\pm}\cap D_{\pi'_2}^\pm ,1}\simeq M^{\pi'_1}_{D_{\pi'_1}^{\pm},1}\otimes_{\C[z]}M^{\pi'_2}_{D_{\pi'_2}^{\pm},1}$$
from which relation (iii) follows.

If $\pi_1$, $\pi_2$ are two paths, with $\pi_1\geq\pi_2$, then $D_{\pi_1}^+\cap D_{\pi_2}^-=\emptyset$, hence 
$$[M^{\pi_1}_{D_{\pi_1}^+,1}]\cdot[M^{\pi_2}_{D_{\pi_2}^-,1}]=[M^{\pi_1}_{D_{\pi_1}^+,1}\otimes_{\C[z]}M^{\pi_2}_{D_{\pi_2}^-,1}]=0,$$ 
which shows that relation (iv) is satisfied. 

For relation (v), let $\omega=\pi_1+\pi_2=\pi_3+\pi_4$ and notice that the module
$$ (M^{\pi_1}_{D_{\pi_1}^+,1}\otimes_{\C[z]}M^{\pi_2}_{D_{\pi_2}^-,1})\oplus (M^{\pi_3}_{D_{\pi_3}^+,1}\otimes_{\C[z]}M^{\pi_4}_{D_{\pi_4}^-,1})$$
has support $D^+_{\pi_1\wedge\pi_3}\cap D^-_{\pi_2\vee\pi_4}$ and all the weight spaces are of dimension one, except the ones in $D^+_{\pi_1\vee\pi_3}\cap D^-_{\pi_2\wedge\pi_4}$ which have dimension two. It follows then from Theorem \ref{thm:simples-cyl} that
\begin{align*} (M^{\pi_1}_{D_{\pi_1}^+,1}\otimes_{\C[z]}M^{\pi_2}_{D_{\pi_2}^-,1})\oplus (M^{\pi_3}_{D_{\pi_3}^+,1}\otimes_{\C[z]}M^{\pi_4}_{D_{\pi_4}^-,1}) & \simeq (M^{\pi_1\wedge\pi_3}_{D^+_{\pi_1\wedge\pi_3},1}\otimes_{\C[z]}M^{\pi_2\vee\pi_4}_{D_{\pi_2\vee\pi_4}^-,1})\oplus (M^{\pi_1\vee\pi_3}_{D_{\pi_1\vee\pi_3}^+,1}\otimes_{\C[z]}M^{\pi_2\wedge\pi_4}_{D_{\pi_2\wedge\pi_4}^-,1}) \\
& \simeq  (M^{\pi'_1}_{D_{\pi'_1}^+,1}\otimes_{\C[z]}M^{\pi'_2}_{D_{\pi'_2}^-,1})\oplus (M^{\pi'_3}_{D_{\pi'_3}^+,1}\otimes_{\C[z]}M^{\pi'_4}_{D_{\pi'_4}^-,1})
\end{align*}
for any $\pi'_1,\pi'_2,\pi'_3,\pi'_4$ such that $\pi'_1+\pi'_2=\pi'_3+\pi'_4=\omega$ and $\pi'_1+\pi'_3=\pi_1+\pi_3$, because then we have $\pi'_1\vee\pi'_3=\pi_1\vee\pi_3$, $\pi'_1\wedge\pi'_3=\pi_1+\pi_3$ and similarly for $\pi'_2$ and $\pi'_4$. Thus the relation is verified.

If $\pi_1$ and $\pi_2$ are intersecting paths, then $D_{\pi_1}^+\cap D_{\pi_2}^-$ is a union of contractible components, hence
\begin{align*}[M^{\Bzr}_{C,\xi}]\cdot [M^{\pi_1}_{D_{\pi_1}^+,1}]\cdot[M^{\pi_2}_{D_{\pi_2}^-,1}]&=[M^{\Bzr}_{C,\xi}]\cdot\sum_{D''\in H_0(D_{\pi_1}^+\cap D_{\pi_2}^-)}[M^{\pi_1+\pi_2}_{(D'',0)}]\\
&=\sum_{D''\in H_0(D_{\pi_1}^+\cap D_{\pi_2}^-)} [M^{\Bzr}_{C,\xi}]\cdot[M^{\pi_1+\pi_2}_{(D'',0)}] \\
&= \sum_{D''\in H_0(D_{\pi_1}^+\cap D_{\pi_2}^-)}[M^{\pi_1+\pi_2}_{(D'',0)}]\\
&= [M^{\pi_1}_{D_{\pi_1}^+,1}]\cdot[M^{\pi_2}_{D_{\pi_2}^-,1}]
\end{align*}
and relation (vi) is satisfied.

Finally, we want to prove that the map is injective by showing that the image of a basis of the polynomial ring maps to linearly independent elements. A general element in the polynomial ring can be written, using relation (ii), as a linear combination of monomials such as
\begin{equation}\label{eq:polysum} \gamma_{\xi} (x^-_{\pi_{1}})^{a_{\pi_{1}}}\cdots (x^-_{\pi_{k}})^{a_{\pi_{k}}}(x^+_{\nu_{1}})^{b_{\nu_{1}}}\cdots (x^+_{\nu_{s}})^{b_{\nu_{s}}}.\end{equation}
But then, using relations (iii) repeatedly, we can assume that $\pi_{1}>\cdots>\pi_{k}$ and $\nu_{1}>\cdots>\nu_{s}$. Further, by applying \eqref{eq:wedge-vee} and again (iii), as many times as necessary, we can obtain that
$\pi_{1}>\cdots>\pi_{k-1}>\pi_{k}>\nu_{2}\cdots>\nu_{s}$ and $\pi_{1}>\cdots>\pi_{k-1}>\nu_{1}>\nu_{2}\cdots>\nu_{s}$. This is because if, for example, $\pi_{k}\not\geq\nu_{2}$, then
\begin{align*} x^-_{\pi_{k}}x^+_{\nu_{1}}x^+_{\nu_{2}}&=(x^-_{\pi_{k}}x^+_{\nu_{2}})x^+_{\nu_{1}}\\
\text{ ( by (vii) ) }&=(x^-_{\pi'}x^+_{\nu'}-x^+_{\pi_{k}}x^-_{\nu_{2}})x^+_{\nu_{1}} \\
\text{ with $\pi'=\pi_{k}\vee\nu_{2}$ and $\nu'=\pi_{k}\wedge\nu_{2}$} & \\
&= x^-_{\pi'}x^+_{\nu'}x^+_{\nu_{1}}-x^+_{\pi_{k}}x^-_{\nu_{2}}x^+_{\nu_{1}} \\
\text{ ( by (iv) since $\nu_{1}\geq\nu_{2}$ ) }&= x^-_{\pi'}x^+_{\nu'}x^+_{\nu_{1}} \\
&= x^-_{\pi'}x^+_{\nu_{1}}x^+_{\nu'}\\
\text{ ( by (iii) ) }&=x^-_{\pi'}x^+_{(\nu_{1}\vee\nu')}x^+_{(\nu_{1}\wedge\nu')}
\end{align*}
and now we have both $\pi'\geq\nu'\geq (\nu_{1}\wedge\nu')$ and $(\nu_{1}\vee\nu')\geq (\nu_{1}\wedge\nu')$.
In general then, we can always reduce to monomials in \eqref{eq:polysum} such that only the paths $\pi_{k}$ and $\nu_{1}$ can potentially cross. 

In the case where $\pi_{k}>\nu_{1}$ and the two paths are disjoint, then $D^-_{\pi_{k}}\cap D^+_{\nu_{1}}$ is connected and noncontractible, and we have
$$\alpha\left(\gamma_{\xi} (x^-_{\pi_{1}})^{a_{\pi_{1}}}\cdots (x^-_{\pi_{k}})^{a_{\pi_{k}}}(x^+_{\nu_{1}})^{b_{\nu_{1}}}\cdots (x^+_{\nu_{s}})^{b_{\nu_{s}}}\right)=\left[M^\omega_{D^-_{\pi_{k}}\cap D^+_{\nu_{1}},\xi}\right]$$
where $\omega=\sum_i a_{\pi_{i}}\pi_{i}+\sum_j b_{\nu_{j}}\nu_{j}$, hence the images of different monomials are linearly independent.

In the case where $\pi_{k}$ and $\nu_{1}$ do intersect, then using relation (vi) we have $\gamma_{\xi}(x^-_{\pi_{k}})^{a_{\pi_{k}}}(x^+_{\nu_{1}})^{b_{\nu_{1}}}=(x^-_{\pi_{k}})^{a_{\pi_{k}}}(x^+_{\nu_{1}})^{b_{\nu_{1}}}$. Now, suppose that there are $m$ bounded connected components of $C\setminus\Supp(\pi_k+\nu_1)$, and for each $\underline{i}=(i_1,\ldots,i_m)\in\{0,1\}^m$, let $\pi_{\underline{i}}$ be the unique path that passes above the $j$-th component if $i_j=1$ and below the $j$-th component if $i_j=0$, for all $j=1,\ldots m$. For $\underline{i}\in\{1,0\}^m$, let $|\underline{i}|=\sum_{j=1}^m i_j$ and $(\underline{i})'=(1-i_1,\ldots,1-i_m)$ be the `opposite' sequence. Then for all $\underline{i}\in\{0,1\}^m$ we have that $\pi_k+\nu_1=\pi_{\underline{i}}+\pi_{\underline{i}'}$. We claim that any monomial of the form $x^-_{\pi_{\underline{i}}}x^+_{\pi_{\underline{i}'}}$ can be written as a linear combination of $x^-_{\pi_{\underline{j}}}x^+_{\pi_{\underline{j}'}}$, with $|j|=1$. We proceed by induction on $|i|$. If $|i|=0$, then $\pi_{\underline{i}'}>\pi_{\underline{i}}$ so $x^-_{\pi_{\underline{i}}}x^+_{\pi_{\underline{i}'}}=0$ by (iv). If $|i|=1$, the result is clear. Now suppose that $|i|\geq 2$, and that $i_\ell=1$ for some $\ell$. Define $\underline{i}^\ell=(i_1,\ldots, i_{\ell-1},0,i_{\ell+1},\ldots, i_m)$ and $\underline{\ell}=(0,\ldots,0,\stackrel{\ell}{1},0,\ldots,0)$. Then
\begin{align*}
x^-_{\pi_{\underline{i}}}x^+_{\pi_{\underline{i}'}}&=x^-_{\pi_{\underline{i}}}x^+_{\pi_{\underline{i}'}}+x^-_{\pi_{(0,\ldots,0)}}x^+_{\pi_{(1,\ldots,1)}}\\
\text{ ( by (v) ) } &= x^-_{\pi_{\underline{i}^\ell}}x^+_{\pi_{(\underline{i}^\ell)'}}+x^-_{\pi_{\underline{\ell}}}x^+_{\pi_{\underline{\ell}'}}
\end{align*}
and by inductive hypothesis, since $|\underline{i}^\ell|=|\underline{i}|-1$, $x^-_{\pi_{\underline{i}^\ell}}x^+_{\pi_{(\underline{i}^\ell)'}}$ can be written as a linear combination as desired, so the claim follows.

Then, any monomial containing $x^-_{\pi_{k}}x^+_{\pi_{\nu_1}}$ can be written as a linear combination of terms containing $x^-_{\pi_{\underline{i}}}x^+_{\pi_{\underline{i}'}}$, with $|\underline{i}|=1$. These map under $\alpha$ to $[M^\omega_{D,0}]$ with $D$ a contractible connected component of $C\setminus\Supp(\omega)$ which are linearly independent elements of $\scA(\Omega',\CO)$, showing that $\alpha$ is indeed injective, hence an isomorphism.
\end{proof}

\bibliographystyle{siam} 

\end{document}